\documentclass[reqno,12pt]{amsart}
\usepackage{amsmath, latexsym, amsfonts, amssymb, amsthm, amscd}
\usepackage{mathrsfs,enumerate}
\usepackage{graphicx}
\usepackage{url}

\numberwithin{equation}{section}

\newtheorem{theorem}{Theorem}[section]
\newtheorem{lemma}[theorem]{Lemma}
\newtheorem{prop}[theorem]{Proposition}
\newtheorem{cor}[theorem]{Corollary}
\newtheorem{rem}[theorem]{Remark}
\newtheorem{definition}[theorem]{Definition}

\newcommand{\N}{{\ensuremath{\mathbb N}} }

\newcommand{\bbP}{{\ensuremath{\mathbb P}} }

\newcommand{\R}{{\ensuremath{\mathbb R}} }

\newfont{\indic}{bbmss12}

\def\un#1{\hbox{{\indic 1}$_{#1}$}}

\newcommand\eps{\epsilon}

\renewcommand{\a}{\alpha}
\renewcommand{\b}{\beta}
\renewcommand{\k}{\kappa}
\newcommand{\dt}{\frac{\mathrm{d}}{\mathrm{d}t}}
\renewcommand{\d}[2]{\dfrac{\partial #1}{\partial #2}}
\newcommand{\dplus}[2]{\dfrac{\partial^+ #1}{\partial #2}}
\newcommand{\ddt}[1]{\dfrac{\mathrm{d} #1}{\mathrm{d}t}}
\newcommand{\dplusdt}[1]{\dfrac{\mathrm{d}^+ #1}{\mathrm{d}t}}

\newcommand{\la}{\langle}
\newcommand{\ra}{\rangle}
\newcommand{\ds}{\: \mathrm{d}s}
\newcommand{\dl}{\: \mathrm{d}\ell}
\newcommand{\bsl}{\backslash}

\newcommand{\Mf}{\mathcal{M}_f^+}
\newcommand{\Mcc}{\mathcal{M}_c^+}
\newcommand{\pinf}{+ \infty}
\newcommand{\Dphi}{\Delta \phi(m,m')}

\newcommand{\dm}{\mathrm{d}m}
\newcommand{\Cc}{C_c(0,\pinf)}

\newcommand{\G}{\Gamma}

\newcommand{\cinf}{c_{\infty}}

\newcommand{\binf}{\beta_{\infty}}

\newcommand{\Tgel}{T_{\rm gel}}

\begin{document}

\title[Post-gelation uniqueness for coagulation equations]{Uniqueness
of post-gelation solutions of a class of coagulation equations}

\author{Raoul Normand}
\address{Laboratoire de Probabilit\'es et Mod\`eles Al\'eatoires, Universit\'e Paris 6 -- Pierre et Marie Curie, 175 rue du Chevaleret, 75013 Paris, France}
\email{raoul.normand@upmc.fr}
\author{Lorenzo Zambotti}
\address{Laboratoire de Probabilit\'es et Mod\`eles Al\'eatoires, Universit\'e Paris 6 -- Pierre et Marie Curie, 175 rue du Chevaleret, 75013 Paris, France}
\email{lorenzo.zambotti@upmc.fr}



\begin{abstract}
We prove well-posedness of global solutions
for a class of coagulation equations which exhibit
the gelation phase transition. To this end, we solve an associated partial differential
equation involving the generating functions before and after the phase transition. Applications
include the classical Smoluchowski and Flory equations with
multiplicative coagulation rate and the recently introduced symmetric model
with limited aggregations. For the latter, we compute the limiting concentrations and we relate them to random graph models.
\end{abstract}

\maketitle

\section{Introduction}

\subsection{Coagulation models}
In this paper we deal with the problem of uniqueness of post-gelation solutions of
several models of coagulation, namely Smoluchowski's and Flory's classical models, and the corresponding
models with limited aggregations recently introduced by Bertoin \cite{BertoinTSS}. 

Smoluchowski's coagulation equations
describe the evolution of the concentrations of particles in a system where particles
can perform pairwise coalescence, see e.g. \cite{AldousReview, LauMi,NorrisReview}. In the original model of Smoluchowski \cite{Smolu}, a pair of particles
of mass, respectively, $m$ and $m'$, coalesce at rate
$\k(m,m')$ and produce a particle of mass $m+m'$.
In the discrete setting, the evolution of the concentration $c_t(m)$ of particles of
mass $m\in\N^*$ at time $t \geq 0$ is given by the following system
\begin{equation}\label{coagu}
\dt c_t(m) = \frac12 \sum_{m'=1}^{m-1} \k(m,m')\,  c_t(m') \, c_t(m-m')  - \sum_{m'\geq 1} \k(m,m') c_t(m)c_t(m').
\end{equation}

Norris considered in \cite{NorrisCC} far more general models of {\it cluster coagulation}, 
where the rate of coalescence does not depend only on the mass of the particles but
also on other parameters. In this general setting, most results on existence and uniqueness are obtained before a critical time, known as the {\it gelation time}, while the global behavior of the solutions after gelation, and in particular uniqueness, is not known. 

An example of a solvable cluster coagulation model is Bertoin's {\it model with limited aggregations} \cite{BertoinTSS}, which we shall simply call the {\it model with arms}. In this
case, particles have a mass but also carry a certain number of potential links, called {\it arms}. Two particles of mass $m$ and $m'$ may coagulate only if they have a positive number of arms, say $a$ and $a'$. When they coagulate, an arm of each is used to create the bond and both arms are then deactivated, hence creating a particle with $a+a'-2$ arms and mass $m+m'$. The coagulation rate of these two particles is $aa'$.
Therefore, if $c_t(a,m)$ is the concentration of particles with $a\in\N=\{0,1,\dots\}$ arms and mass $m\in\N^*=\{1,2,\dots\}$, then
the coagulation equation reads
\begin{equation} \label{smolubras0}
\begin{split}
\dt c_t(a,m) & = \frac12 \sum_{a'=1}^{a+1} \sum_{m'=1}^{m-1} a'(a+2-a') c_t(a',m')c_t(a+2-a',m-m') \\
 & \quad - \sum_{a' \geq 1} \sum_{m'\geq 1} aa'c_t(a,m) c_t(a',m').
\end{split}
\end{equation}
For monodisperse initial concentrations, i.e. $c_0(a,m)=\un{\{m=1\}}\mu(a)$,
with $\mu=(\mu(a))_{a\in\N}$ a measure on $\N$ with unit mean,
it is proved in \cite{BertoinTSS} that this equation has a unique solution on some interval $[0,T)$, where $ T = \pinf$ if and only if $K \leq 1$, where
\begin{equation} \label{K}
K:=\sum_{a\geq 1} a(a-1)\mu(a).
\end{equation}
In other words, if particles at time 0 have, on average, few arms, equation \eqref{smolubras0} has a unique solution defined for all $t \geq 0$. 
When this is the case, as time passes, all available arms are used to create bonds and only
particles with no arms remain in the system. The limit concentrations $\cinf(0,m)$
as $t\to+\infty$ of such particles turn out to be related
to the distribution of the total population generated by a sub-critical Galton-Watson
branching process (see e.g. \cite{AtNey}) started from two ancestors: see \cite{BertoinTSS,BertoinSido}
and section \ref{limconc} below.

\subsection{The gelation phase transition}
A formal computation shows that solutions
of \eqref{coagu} with multiplicative kernel $\k(m,m')=mm'$ should have constant mass
\begin{equation} \label{Mt0}
M_t := \sum_{m \geq 1} m c_t(m)
, \qquad t\geq 0,
\end{equation}
i.e. $\frac d{dt}M_t=0$. It is however well-known that if large particles can coagulate sufficiently fast, then
one may observe in finite time a phenomenon called {\it gelation}, namely the formation of particles of infinite mass, the {\it gel}. These particles do not count in the computation of the mass
so from the gelation time on, $M_t$ starts to decrease.

The reason why \eqref{smolubras0} can be solved, is that it can be transformed
into a solvable PDE involving the generating function of $(c_t)_{t\geq 0}$. 
In Equation \eqref{coagu}, this transformation is also possible for several
particular choices of the kernel $\k(m,m')$, namely
when $\k$ is constant, additive or multiplicative: see e.g. \cite{DeaconuTanre}. In the multiplicative case $\k(m,m')=mm'$, which is our main concern here, 
the total mass 
is a parameter of \eqref{coagu} and of the associated PDE, which is therefore easy to solve only when $(M_t)_{t\geq 0}$ is known.  Existence and uniqueness of solutions of \eqref{coagu} are thus
easy up to gelation, since in this regime, the total mass $M_t$ is constant.

After gelation, the gel may or may not interact with the other particles. If it does, Equation \eqref{coagu} has to be modified into Flory's equation \eqref{flory}. Else, the gel is inert, in which case Smoluchowski's equation continues to hold. Obviously, they are identical before gelation.

Occurrence of gelation depends heavily on the
choice of the coagulation rate $\k(m,m')$, and in the multiplicative case, gelation always occurs \cite{EscobedoGaMC,FouLauWP,LauCCC}. After gelation, the mass is not known, so $M_t$ itself becomes an unknown of the equation, and well-posedness of the equation is then much less trivial. The multiplicative kernel is therefore particularly interesting, since it exhibits a non-trivial behavior but can still be studied in detail by means of explicit computations. 

\medskip
The same phenomenon of gelation has been observed in \cite{BertoinTSS} for \eqref{smolubras0} for monodisperse initial concentrations $c_0$.
A formal computation shows that the the mean number of arms $A_t$ 
\[
A_t:=\sum_{a,m \geq 1} a\, c_t(a,m), \qquad t\geq 0,
\]
satisfies the equation $\frac d{dt}A_t=-A_t^2$ and should therefore be equal to
$\frac{1}{1+t}$ for all $t\geq 0$. In fact, this explicit expression
holds only until a critical time, which is shown to be equal
to $1/(K-1)$ if $K>1$ and to $+\infty$ if $K\leq 1$, where $K$ is defined in \eqref{K}.
Again, the associated PDE is easy to solve before gelation since then, $A_t$ is known, while afterwards, the PDE contains
the unknown parameter $A_t$.

\subsection{Main result}
In this paper we investigate the global behavior of Smoluchow\-ski's equation with arms \eqref{smolubras0} before, at and
after the gelation phase transition, proving existence and uniqueness of global solutions
for a large class of initial conditions. 
The technique used, as in \cite{BertoinTSS}, is to transform the equation into a PDE. Since the total number of arms $(A_t)_{t\geq 0}$ is not a priori known, this PDE is non local, unlike the one obtained in the regime before gelation. This is the main difficulty we have to deal with. We use a modification of the classical method of characteristics to show uniqueness of solutions to this PDE, and hence to \eqref{smolubras0}. We can consider initial conditions $(c_0(a,m), a \in \N, m \in \N^*)$ with an initial infinite number of arms, that is, such that
\[
A_0 := \sum_{a, m \geq 1} a c_0(a,m)
\]
is infinite, and show that there is a unique solution ``coming down from infinity
sufficiently fast'', i.e. such that, for positive $t$,
\[
\int_0^t A_s^2 \ds < \pinf.
\]
Note however that this is no technical condition, but a mere assumption to ensure that the equation is well-defined.

We also consider a modification of this model which corresponds to Flory's equation for the model with arms. In this setting, the infinite mass particles, that is, the gel, interact with the other particles. We also prove existence, uniqueness and study the behavior of the solutions for this model.

In both cases, our technique provides a representation formula allowing to compute various quantities, as the mean number of arms in the system and the limiting concentrations. In Flory's case, we extend to all possible initial concentrations the computations done in \cite{BertoinTSS} in absence of gelation. In the first model, a slight modification appears which calls for a probabilistic interpretation; see
section \ref{limconc} below.

This seems to be the first case of a cluster coagulation model for which global well-posedness
in presence of gelation can be proven. Another setting to which these techniques could be applied
is the {\it coagulation model with mating} introduced in \cite{NorMFC}.

\subsection{Limiting concentrations}\label{limconc}
In \cite{BertoinTSS}, explicit solutions to \eqref{smolubras0} are given for monodisperse initial conditions $c_0(a,m) = \mu(a) \un{\{m=1\}}$ for some measure $\mu$ on $\N$ with unit first moment. In particular, when there is no gelation, i.e. $K\leq 1$ where $K$ is as in \eqref{K}, and $\mu \neq \frac12\delta_2$, there are limiting concentrations
\[
\cinf(a,m) = \frac{1}{m(m-1)} \, \nu^{*m}(m-2)\, \un{\{a=0\}}, \qquad m \geq 2,
\]
where $\nu(m) = (m+1) \mu(m+1)$ is a probability measure on $\N$ different from $\delta_1$. This formula clearly resembles the well-known formula of Dwass \cite{Dwass}, which provides the law of the total progeny $T$ of a Galton-Watson process with reproduction law $\nu$, started from two ancestors:
\[
\bbP(T=m)=\frac2m \, \nu^{*m}(m-2), \quad m \geq 2.
\]
The similarity between the two formulas is no coincidence and is explained in \cite{BertoinSido} by means of the configuration model. For basics on Galton-Watson processes, see e.g. \cite{AtNey}.

Let us briefly explain the result of \cite{BertoinSido}, referring e.g. to \cite{vdHofstad} for more results on general random graphs. The configuration model aims at producing a random graph whose vertices have a prescribed degree. To this end, consider a number $n$ of vertices, each being given independently a number of arms (that is, half-edges) distributed according to $\mu$. Then, two arms in the system are chosen uniformly and independently, and form an edge between the corresponding vertices. This procedure is repeated until there are no more available arms. Hence, one arrives to a final state which can be described as a collection of random graphs. Then Corollary 2 in \cite{BertoinSido} and the discussion below show that, when there is no gelation, the proportion of trees of size $m$ tends to $\cinf(0,m)$ when the number $n$ of vertices tends to infinity. Hence, the final states in the configuration model and in Smoluchowski's equation with arms coincide. This shows that the former is a good discrete model for coagulation.

Interestingly, the absence-of-gelation condition $K\leq 1$ is equivalent to (sub)-criticality
of the Galton-Watson branching process with reproduction law $\nu$, i.e. to almost sure extinction of the progeny, while $K>1$ and gelation at finite time are equivalent to super-criticality of the GW process.

In this paper, we obtain the limiting concentrations for \eqref{smolubras0} and its modified version when there is gelation. Let us start with the modified model, which is the counterpart of Flory's equation for the model with arms. In this case, and with the same notations as above, we obtain the limit concentrations
\[
\cinf(a,m) = \frac{1}{m(m-1)} \, \nu^{*m}(m-2)\, \un{\{a=0\}}, \qquad m \geq 2,
\]
that is, the same explicit form as the one obtained in absence of gelation. 
Again, this formula can be interpreted both in terms of a configuration model and
of a super-critical Galton-Watson branching process. The relation between 
Flory's equation with arms and the configuration model is natural, since
in both cases all particles interact with each other, no matter what their size is.
It is worth noticing that, even though the limit concentrations have the same form with or without gelation, still some mass is eventually lost in presence of gelation, see \eqref{Minfty} below.

We also obtain the limiting concentrations for Smoluchowski's equation with arms, namely
\[
\cinf(a,m) = \frac{1}{m(m-1)} \, \binf^{m-1} \nu^{*m}(m-2)\, \un{\{a=0\}},
\]
where $\binf$ is some constant, which is $1$ when there is no gelation, and is greater than $1$ otherwise, see Section \ref{seccinfsm}. However, the probabilistic interpretation of $\binf$ is unclear. One can recover Smoluchowski's equation with arms from discrete models by preventing big particles from coagulating, as is done in \cite{FouLauMLP} for the standard Smoluchowski equation, but the precise meaning of $\binf$ still seems to require some labor.

\subsection{Bibliographical comments}
Smoluchowski's equation \eqref{coagu} has been extensively studied; we refer to the reviews \cite{AldousReview,LauMi,NorrisReview}. Conditions on the kernel $\k$ are know for absence or presence of gelation, though this requires a precise definition of gelation, see e.g. \cite{EscobedoGiC}, or \cite{Jeon} in a probabilistic setting. For a general class of kernels Smoluchowski's solution has a unique solution before gelation \cite{NorrisReview,DuboStew,FouLauWP,LauMi}, and in the multiplicative case gelation always occurs \cite{EscobedoGaMC,FouLauWP,LauCCC}. 

For the monodisperse initial condition $c_0(m)=\un{\{m=1\}}$, the first proof of existence and uniqueness to \eqref{coagu} before gelation is given in \cite{McLeod}, and a proof of global existence and uniqueness can be found in \cite{Kokholm}. The case of general nonzero initial conditions has been considered by several papers in the Physics literature \cite{ErnstKoG, ErnstCP,LeyvTschu,StellZiff,ZiffKoP}, and by at least one mathematical paper \cite{ShirRoesSR}, which however treats in full details only the regime before gelation, see Remark \ref{see} below. The same authors also provide in \cite{ShirRoesPGM} an exact formula for the post-gelation mass of \eqref{coagu}, but with no rigorous proof. 

Thus, a clear statement about well-posedness of \eqref{coagu} for the most general initial conditions still seems to be missing, and our paper tries to fill this gap. We adapt the classical method of characteristics for generating functions, see \cite{DeaconuTanre,BertoinTSS}, which yields easily uniqueness before gelation for a multiplicative kernel \cite{MenonPego}. We can in particular consider initial concentrations with infinite total mass, i.e. such that
\[
M_0:=\int_{(0,+\infty)} m\, c_0(\dm)=+\infty,
\]
as long as $\int_{(0,+\infty)} (m\wedge 1)\, c_0(\dm)<+\infty$.
This covers for instance initial conditions of the type $c_0(\dm)=C_p\, m^{-p}\,
\dm$ with $p\in[1,2)$. 

Our main concern is uniqueness, since existence of solutions has been obtained in a much more general setting by analytic \cite{LauGS,LauCCC,NorrisCC} or probabilistic \cite{FouLauMLP,Jeon} means. However, the case of an infinite initial mass seems to have been considered only in \cite{LauGS} in the discrete case, so we refer to Section \ref{exist} below for a proof.

\subsection{Plan of the article}
We start off in Section 2 by considering existence, uniqueness and representation formulas for global solutions of \eqref{coagu}, introducing and
exploiting all main techniques which are needed afterwards to tackle the same issues in the
case of \eqref{smolubras0}. We prove that for the most general initial conditions $\mu_0(\dm)$, a positive measure on
$(0,+\infty)$, Smoluchowski's equation with a multiplicative kernel has a unique solution
before and after gelation. We also show existence and uniqueness for the modified version of Smoluchowski's model, namely Flory's equation, in Section 3. The techniques used are generalized in Section 4 and 5, where we prove analogous results for the models with arms. We compute the limiting concentrations in Section 6, which are not trivial, in comparison with the standard Smoluchowski and Flory cases, for which they are always zero. 

\section{Smoluchowski's equation}

In this section we develop our method in the case of equation \eqref{coagu},
proving existence, uniqueness and representation formulas for global solutions.
Let us first fix some notations.
\begin{itemize}
\item $\Mf$ is the set of all non-negative finite measures on $(0,\pinf)$.
\item $\Mcc$ is the set of all non-negative Radon measures on $(0,\pinf)$.
\item For $\mu \in \Mcc$ and $f \in L^1(\mu)$ or $f \geq 0$,
\[
\la \mu, f \ra = \int_{(0,\pinf)} f(m) \: \mu(\dm).
\]
We will write $m$ for the function $m \mapsto m$, $m^2$ for $m \mapsto m^2$, etc.
\item For $\phi : (0,\pinf) \to \R$ and $m,m' > 0$, $\Dphi = \phi(m+m')-\phi(m)-\phi(m')$.
\item $\Cc$ is the space of continuous functions on $(0,\pinf)$ with compact support.
\item For a function $(t,x) \mapsto \phi_t(x)$, $\phi_t'(x)$ is the partial derivative of $\phi$ with respect to $x$.
\item $\dplus{}{t}$ or $\dplusdt{}$ denotes the right partial derivative with respect to $t$.
\end{itemize}
We are interested in Smoluchowski's equation \eqref{coagu} with multiplicative
coagulation kernel $\kappa(m,m')=mm'$. Note that the second requirement in the following definition is only present for the equation to make sense.
\begin{definition}\label{defsolsmolu}
Let $\mu_0\in\Mcc$.
We say that a family $(\mu_t)_{t\geq 0}\subset \Mcc$ solves Smoluchowski's equation if
\begin{itemize}
\item for every $t>0$, $\int_0^t \la \mu_s(\dm), m\ra^2 \ds < \pinf$,
\item for all $\phi \in \Cc$ and $t>0$
\begin{equation}
\la \mu_t , \phi \ra = \la \mu_0, \phi \ra + \frac12 \int_0^t \la \mu_s(\dm) \mu_s(\dm'), mm'\, \Dphi \ra \ds,
\label{smolu2}
\end{equation}
\item if $\la \mu_0 , m^2 \ra < \pinf$, then $t \mapsto \la \mu_t , m^2 \ra$ is bounded in a right neighborhood of 0.
\end{itemize}
\end{definition}
The global behavior of this equation has been studied first for monodisperse initial conditions (i.e. $\mu_0=\delta_1$), in which case it can be proven that there is a unique solution $(\mu_t)_{t\geq 0}$ on $\R^+$, which is also explicit, see \cite{McLeod, Kokholm}. This solution clearly exhibits the gelation phase transition. Up to  the gelation time $\Tgel=1$, the total mass $\la \mu_t, m \ra$ is constant and equal to 1, and then it decreases: $\la \mu_t,m \ra = 1/t$ for $t\geq1$. Moreover, the second moment $\la \mu_t,m^2 \ra$ is finite before time $1$, and then infinite on $[1,+\infty)$.
It is also known in the literature that for any nonzero initial conditions, there is a gelation time $0 < \Tgel  < \pinf$, such that there is a unique solution to \eqref{smolu2} on $[0,\Tgel )$, and $\la \mu_t , m^2 \ra \to \pinf$ when $t \to \Tgel ^-$: see e.g. \cite{FouLauWP}.
\begin{theorem}\label{thsmolu}
Let $\mu_0\in\Mcc$ a non-null measure such that
\begin{equation}\label{condition}
\la \mu_0 , m\wedge 1 \ra =  \int_{(0,\pinf)} (m\wedge 1) \, \mu_0(\dm)<+\infty.
\end{equation}
We can then define
\[
M_0:=\la \mu_0 , m \ra \in (0,\pinf], \qquad K := \la \mu_0 , m^2 \ra  \in (0,\pinf],
\]
and the function
\begin{equation}\label{g_0}
g_0(x):=\la \mu_0, mx^m \ra = \int_{(0,\pinf)} mx^m \: \mu_0(\dm),  \qquad x\in[0,1]
\end{equation}
with $g_0(1)=M_0\in(0,+\infty]$. Let 
\begin{equation} \label{Tgel}
\Tgel := 1/K \in [0,+\infty).
\end{equation}
Then Smoluchowski's equation \eqref{smolu2} has a unique solution on $\R^+$.
It has the following properties.
\begin{enumerate}
\item The total mass $M_t=\la \mu_t, m \ra$ is continuous on $[0,+\infty)$. It is constant on $[0,\Tgel]$ and strictly decreasing on $[\Tgel,\pinf)$. It is analytic on $\R^+ \backslash \{ \Tgel \}$.
\item If the following limit exists
\[
\nu:= - \lim_{x \to 1^-} \frac{\left(g_0'(x)\right)^3}{g_0'(x)+xg_0''(x)} \in[-\infty,0],
\]
then the right derivative $\dot{M}_{\Tgel}$ of $M$ at $t=\Tgel$ is equal to $\nu$.
\item Let $m_0 = \inf \mathrm{supp}\ \mu_0\in[0,+\infty)$. When $t \to \pinf$,
\[
\frac{1}{tM_t} \to m_{0}.
\]
\item The second moment $\la \mu_t, m^2 \ra$ is finite for $t\in[0,\Tgel)$ and infinite for $t\in[\Tgel,+\infty)$.
\end{enumerate}
\end{theorem}

\begin{rem}{\rm
\begin{itemize}
  \item This result allows to recover the pre- and post-gelation formulas obtained with no rigorous proof in some earlier papers \cite{ErnstCP,ErnstKoG,Kokholm,LeyvTschu,ShirRoesSR,ShirRoesPGM,StellZiff}. The decrease of the mass in $1/t$ when $m_0 > 0$ was also observed in these papers. Also, some upper bounds in $1/t$ for the mass were proven in \cite{EscobedoGiC,LauCCC}.
	\item If $m=0$, the mass tends to $0$ more slowly than $1/t$: small particles need to coagulate before any big particle can appear, and they coagulate really slowly. For instance, a straightforward computation shows that if $\mu_0(\dm)=e^{-m} \dm$, then $M_t \sim t^{-2/3}$. More generally, the explicit formula in Proposition \eqref{uniuni} allows to compute $M_t$ for any initial conditions.
	\item With this formula, it is easy to check that $\dot{M}_{\Tgel+}$ can be anything from $- \infty$ to $0$. For instance, $\dot{M}_{\Tgel+}=0$ for $g_0(x)=(1-x) \log(1-x)+x$,  $\dot{M}_0=- \infty$ for $g_0(x)=\sqrt{1-x} \log(1-x)+x$, and for $0 < \a < \pinf$, $\dot{M}_0=- \a$ for $g_0(x)=1 - \sqrt{1-x^{2 \a}}$. In particular, $M$ need not be convex on $[\Tgel,+\infty)$.
\end{itemize}
}
\end{rem}

\subsection{Preliminaries}

Let $\mu_0$ be defined as in the previous statement. 
We shall prove that, starting from $\mu_0$, there is a unique solution to \eqref{smolu2} on $\R^+$, and give a representation formula for this solution. This allows to study the behavior of the moments. Let us start with some easy lemmas. So take a solution $(\mu_t)_{t\geq 0}$ to \eqref{smolu2} and set
\begin{equation}
M_t=\la \mu_t , m \ra.
\label{Mt}
\end{equation}
The two following lemmas are easy to prove, using monotone and dominated convergence.
\begin{lemma}\label{lemM}
$(M_t)_{t\geq 0}$ is monotone non-increasing and right-conti\-nuous. Moreover, $M_t<+\infty$ for all $t>0$.
\end{lemma}
\begin{proof}
Take $\phi^K(m)=m$ for $m \in [0,K]$, $\phi^K(m)=2K - m$ for $m \in [K,2K]$, and $\phi^K(m)=0$ for $m \geq 2K$, so that $\phi^K \in C_c$. Plugging $\phi^K$ in Smoluchowski's equation \eqref{smolu2}, letting $K \to \pinf$ and using Fatou's lemma readily shows that $(M_t)_{t\geq 0}$ is monotone non-increasing. Note also that $t\mapsto M_t = \sup_K \la \mu_t , \phi^K \ra$ is the supremum of a sequence of continuous functions and so is lower semi-continuous, which implies, for a monotone non-increasing function, right-continuity. Finiteness of $M_t$
is now obvious since $s \mapsto M_s^2$, and hence $s \mapsto M_s$, are integrable by Definition \ref{defsolsmolu}.
\end{proof}
\begin{lemma}
Assume that $t \mapsto \la \mu_t , m^2 \ra$ is bounded on some interval $[0,T_0]$. Then $M_t = M_0$ for $t \in [0,T_0]$.
\label{lemmt}
\end{lemma}
By Lemma \ref{lemM}, $\la \mu_t , m \ra<+\infty$ for $t>0$, so that we can define
\begin{equation}\label{g}
g_t(x)=\la \mu_t, mx^m \ra = \int_{(0,\pinf)} mx^m \: \mu_t(\dm),  \qquad x\in[0,1], \ t>0,
\end{equation}
which is the generating function of $m\mu_t(\dm)$. 
Then, using a standard approximation procedure, it is easy to see that $g$ satisfies
\begin{equation}
\left \{
\begin{array}{ll}
{\displaystyle g_t(x) =g_0(x)+\int_0^t x(g_s(x)-M_s)\dplus{g_s}{x}(x)\ds, \quad t \geq 0, \ x\in(0,1)} \\
g_t(1) = M_t, \quad t\geq 0. \rule{0cm}{0.5cm}
\end{array}
\right.
\label{PDE}
\end{equation}
It is well-known, and will be proven again below, that $M_t=M_0$ for all $t\leq \Tgel$, since then, the PDE \eqref{PDE}
can be solved by the method of characteristics: the function $\phi_t(x):[0,1]\mapsto[0,1]$
\[
\phi_t(x)=xe^{t (M_0-g_0(x))}, \qquad x\in [0,1], \, t\leq \Tgel
\]
is one-to-one and onto, has an inverse $h_t:[0,1]\mapsto[0,1]$ and we find
\[
g_t(x) = g_0(h_t(x)), \qquad x\in[0,1], \, t \leq \Tgel.
\]
However $M_t$ is not necessary constant for $t>\Tgel$ and the form
of $\phi_t$ has to be modified; we thus define
\begin{equation}
\phi_t(x)=x \a_t e^{-t g_0(x)}, \qquad x\in [0,1], \, t>0 
\label{phit}
\end{equation}
where
\begin{equation}
\a_t:=\exp \left(\int_0^t M_s \ds\right), \qquad t\geq 0.
\label{alphat}
\end{equation}
For $t>\Tgel$, $M_t$ is possibly less than $M_0$ and  $\phi_t$, which depends
explicitly on $(M_s)_{s\in[0,t]}$, is possibly neither injective nor surjective. We shall prove
that it is indeed possible to find $\ell_t\in(0,1)$ such that $\phi_t(x):[0,1]\mapsto[0,\ell_t]$ is one-to-one
and $\ell_t$ is uniquely determined by $g_0$.

\subsection{Uniqueness of solutions}
Using an adaptation of the method of characteristics,
we are going to prove the following result. Note that in \cite{ShirRoesSR}, this properties are claimed to be true but a proof seems to lack. We will use the same techniques in the proof of Theorem \ref{thsmolubras} for the model with arms, but they are easier to understand in the present case. 
\begin{prop}\label{uniuni} Let $(\mu_t)_{t\geq 0}$ be a solution of Smoluchowski's equation
\eqref{smolu2}.
\begin{enumerate}
\item For all $t\in[0,\Tgel]$, $M_t=M_0=g_0(\ell_t)$, where $\ell_t:=1$. For  all $t>\Tgel$,
$M_t=g_0(\ell_t)$ where $\ell_t\in(0,1)$ is uniquely defined by
\begin{equation}\label{ell}
\ell_t \, g'_0(\ell_t)= \frac 1t.
\end{equation}
Moreover $\ell_t$ and $\phi_t(\cdot)$ satisfy
\begin{equation}\label{uuuu}
\phi_t'(\ell_t)=0, \qquad
\phi_t(\ell_t)=1>\phi_t(x), \quad \forall \, x\in(0,1).
\end{equation}
\item For all $t>0$, the function $\phi_t(\cdot)$ defined in \eqref{phit} has a right inverse
\begin{equation}
h_t:[0,1]\mapsto[0,\ell_t], \qquad \phi_t(h_t(x))=x, \quad x\in[0,1],
\label{ht}
\end{equation}
and
\begin{equation}\label{uuu}
g_t(x) = g_0(h_t(x)), \qquad t>0, \ x\in[0,1].
\end{equation}
\item The functions  $(\ell_t)_{t\geq 0}$ and  $(M_t)_{t\geq 0}$ are continuous.
\item $(\mu_t)_{t\geq 0}$ is uniquely defined by $\mu_0$.
\end{enumerate}
\end{prop}
\begin{figure}[htb]
\begin{tabular}{cc}
\includegraphics[width=7.1cm]{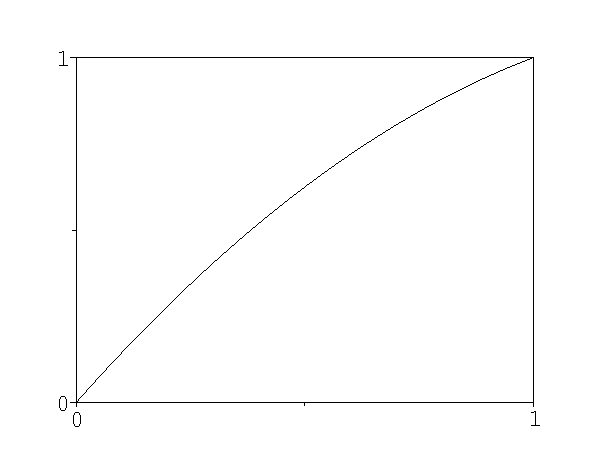}& \includegraphics[width=7.1cm]{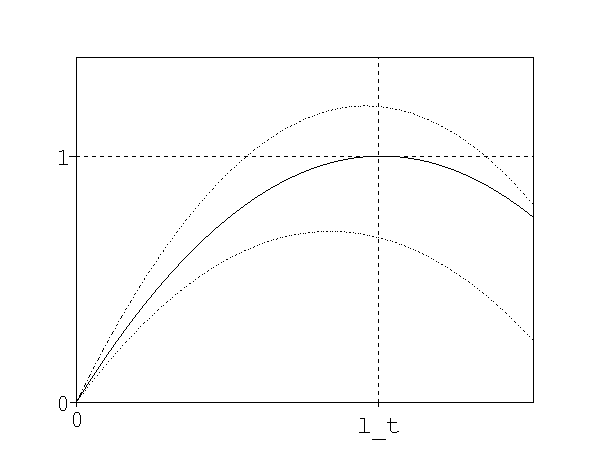}
\end{tabular}
\caption{$\phi_t$ before and after gelation. The dotted lines represent what $\phi_t$ may look like. The solid one is the actual $\phi_t$.}
\end{figure}
\begin{rem}\label{see}{\rm 
\begin{itemize}
  \item For all $t\leq \Tgel$, $M_t=M_0$, $\ell_t=1$ and $\phi_t:[0,1]\mapsto[0,1]$ is
one-to-one and onto. The first thing one needs to prove
is that for all $t>\Tgel$, $\ell_t<1$, i.e. there is indeed $x\in[0,1]$
such that $\phi_t(x)=1$, see Lemma \ref{ti}; the second one, is that $\ell_t=m_t$, i.e. $\phi_t(\cdot)$ has an absolute maximum
at $\ell_t$, see Lemma \ref{lemlt2}. In other words, one has to exclude the dotted lines as possible profiles of
$\phi_t(\cdot)$ in Figure 1. These properties are not obvious, since
$\phi_t$ depends on $(M_s)_{s\in[0,t]}$ which is, at this point, unknown.
All other properties are derived from these two. 
  \item In \cite[Section 6]{ShirRoesSR} one finds a discussion
of post-gelation solutions, in particular of the results of our Proposition \ref{uniuni}.
However this discussion falls short of a complete proof, since the two above-mentioned properties 
are not proven. In particular, no precise statement about
what initial conditions can be considered is given.
\end{itemize}
}
\end{rem}
\noindent
The following lemma is a list of obvious but useful properties satisfied by $g$ and $\phi$.
\begin{lemma}
\label{defsolpde}
The function $g$ defined in \eqref{g} satisfies the following properties.
\begin{itemize}
\item[(a1)] $(t,x) \mapsto g_t(x)$ is finite and continuous on $[0,+\infty)\times [0,1)$;
\item[(a2)] For all $x \in [0,1)$, $t \mapsto g_t(x)$ is right differentiable on $(0,+\infty)$;
\item[(a3)] For all $t \geq 0$, $x \mapsto g_t(x)$ is analytic on $(0,1)$ and monotone non-decreasing;
\item[(a4)] For all $t > 0$, $x \mapsto g_t(x)\in[0,+\infty]$ is continuous on $[0,1]$.
\end{itemize}
The function $\phi$ defined in \eqref{phit} satisfies the following properties.
\begin{itemize}
\item[(b1)] $\phi_t$ is continuous on $[0,1]$ and analytic on $(0,1)$;
\item[(b2)] $\phi_t(0)=0$, $\phi_t(1)=e^{-\int_0^t(M_0-M_s) \ds}\in[0,1]$;
\item[(b3)] $\phi_t'(x)=\a_t e^{-t g_0(x)} (1-txg_0'(x))$ for $x \in (0,1)$;
\item[(b4)] For $t \leq \Tgel$, $\phi_t$ is increasing. For $t > \Tgel$, $x \mapsto x g_0'(x)$ is increasing, $\phi_t'(0)>0$ and $\phi_t'(1) < 0$. In particular, for $t > \Tgel$, there is precisely one point $m_t \in (0,1)$ such that
\begin{equation}\label{mmm}
\phi_t'(m_t)=0.
\end{equation}
\item[(b5)] For $t > \Tgel$, $\phi_t$ is increasing on $[0,m_t]$ and decreasing on $[m_t,1]$.
\end{itemize}
Moreover,
\begin{itemize}
\item[(c1)] The map $(t,x) \mapsto \phi_t(x)$ is continuous on $\R^+ \times [0,1)$;
\item[(c2)] The map $(t,x) \mapsto \phi_t'(x)$ is continuous on $\R^+ \times (0,1)$;
\item[(c3)] For every $x \in [0,1)$, $t \mapsto \phi_t(x)$ is right differentiable and
\begin{equation}\label{phi}
\dplus{\phi_t}{t} = \phi_t(x) (M_t - g_0(x)) \qquad x\in[0,1), \, t\geq 0.
\end{equation}
\end{itemize}
\end{lemma}
Property (b5) implies that there are at most two points in $(0,1)$ where $\phi_t$ equals 1. Take $\ell_t$ to be the smallest, if any, i.e.
\begin{equation}\label{lll}
\ell_t = \inf \{ x \geq 0: \, \phi_t(x)=1 \} \qquad (\inf\emptyset:=1).
\end{equation}
\begin{lemma}\label{ti}
\begin{enumerate}
\item For every $t \geq 0$ and every $x \in [0,\ell_t]$
\begin{equation}
g_t ( \phi_t(x) ) = g_0(x).
\label{gtphit}
\end{equation}
\item For all $t \in [0,\Tgel]$, $\ell_t=1$, and for $t > \Tgel$, $0<\ell_t <1$. In particular,
for all $t>0$, $\phi_t(\ell_t)=1$ and
\begin{equation}
g_0(\ell_t)=g_t(1)=M_t.
\label{g0lt}
\end{equation}
\item Finally, $t\mapsto \ell_t$ is monotone non-increasing and continuous on $\R^+$.
\end{enumerate}
\label{lemlt1}
\end{lemma}

\begin{proof} \begin{enumerate}
\item Let us first prove that there exists $\tau>0$ such that \eqref{gtphit} holds for $t\in[0,\tau[$. Fix $0<a<b<1$. Since $0 < \min_{[a,b]}\phi_0 < \max_{[a,b]}\phi_0 < 1$, then by property (c1) there is $\tau>0$ such that
\[
0 < \min_{[a,b]}\phi_t <  \max_{[a,b]}\phi_t < 1, \qquad \forall\, t \in [0,\tau).
\]
So, for a fixed $x \in [a,b]$, the function
\[
u_t:=g_t(\phi_t(x))-g_0(x)
\]
is well-defined and using \eqref{PDE} and \eqref{phi}, we see that
\begin{align*}
u_t= & \int_0^t \left(\dplus{g_s}{s}(\phi_s(x)) + \d{g_s}{x} (\phi_s(x))
\dplus{\phi_s}{s}(x) \right)\ds = \int_0^t \gamma_s \, u_s \ds
\end{align*}
where
\[
\gamma_t := \d{g_t}{x} (\phi_t(x)) \, \phi_t(x), \qquad t> 0.
\]
Since $x\in[0,1)$, $\sup_{t\in[0,\tau)} |\gamma_t|<+\infty$ and therefore $u_t\equiv 0$. Hence \eqref{gtphit} holds for $x\in[a,b]$ and $t\in[0,\tau[$.
Since both terms of \eqref{gtphit} are analytic functions of $x$ on $(0,\ell_t)$, by analytic continuation, \eqref{gtphit} actually holds on $(0,\ell_t)$, and hence on $[0,\ell_t]$ by continuity.
\item Let us now extend this formula to $t\in \R^+$. Let
\[
T= \sup \{t > 0: \ \forall \, s \in [0,t], \ \forall \, x \in [0,\ell_s], \ g_s(\phi_s(x))=g_0(x) \} \geq \tau > 0,
\]
assume $T < \pinf$, and denote by $\ell$ the left limit of $(\ell_t)_{t\geq 0}$ at $T$. First, $\ell$ cannot be $0$, since otherwise we would get when $s \to T^-$
\[
1 = \phi_s(\ell_s)=\ell_s \a_s e^{-s g_0(\ell_s)} \to 0.
\]
For every $t < T^-$, $0< \ell \leq \ell_t$, so for every $x \in (0, \ell)$, $g_t(\phi_t(x))=g_0(x)$ and $\phi_t(x) < 1$. Using the continuity property (c1) and passing to the limit when $t \to T^-$ in this equality, we get
\[
g_T(\phi_T(x))=g_0(x), \qquad \forall \, x \in (0, \ell).
\]
By the same reasoning as in point (i), we obtain a $T' > T$ such that $g_t(\phi_t(x))=g_0(x)$ for all $t \in [T,T')$ and $x$ in a non-empty open subset of $(0,\ell)$. By analyticity and continuity, the formula $g_t(\phi_t(x))=g_0(x)$ holds for every $t \in [T,T')$ and $x \in [0,\ell_t]$. This contradicts the definition of $T$, and so $T = \pinf$. This concludes the proof of point (1) of the Lemma.
\item For the statement (2) of the Lemma, let us show first that $\la \mu_t , m^2 \ra$ is bounded on $[0,T_0)$, for every $T_0 \in [0,\Tgel)$. Let $T'$, the smallest time when this fails (provided of course that $\Tgel > 0$). By assumption (see Definition \ref{defsolsmolu}), $T' > 0$. Differentiating \eqref{gtphit} with respect to $x$ and having $x$ tend to $\ell_t=1$ gives, for $t < T'$,
\[
g_t'(1)= \la \mu_t , m^2 \ra = \frac{1}{1-tK}.
\]
This quantity explodes only when $t = \Tgel=1/K$, so $T' = \Tgel$.
\item The boundedness of $(\la \mu_t , m^2 \ra)_{t\in[0,T_0)}$ just proven for all $T_0 \in [0,\Tgel)$ and Lemma \ref{lemmt} imply that for $t \in [0,\Tgel)$, $M_t = M_0$. By
the definition \eqref{phit} of $\phi_t$, it follows that $\phi_t(1)=1$ for $t \in [0,\Tgel)$.
But $\phi_t$ is increasing, so $\ell_t = 1$ for $t \in [0,\Tgel)$. Assume now that for some $t > \Tgel$, $\ell_t = 1$. Then \eqref{gtphit} holds on $[0,1]$, and this is impossible since the right term is an increasing function of $x$, whereas the left one decreases in a left neighborhood of 1 since $\phi'_t(1)<0$. The fact that $\phi_t(\ell_t)=1$ follows then directly from the definition of $\ell_t$ and the continuity of $\phi_t(\cdot)$. Finally, the inequality $\ell_t > 0$ is obvious since $\phi_t(0)=0$, and computing \eqref{gtphit} at $x=\ell_t$ gives \eqref{g0lt}. This concludes the proof of (2).
\item We know that $\ell_t=1$ and $M_t=M_0$ for all $t< \Tgel$. Now, $g_0$ is strictly increasing and continuous. Since
$(M_t)_{t\geq 0}$ is monotone non-increasing and right-continuous by Lemma \ref{lemM}, so is $(\ell_t)_{t\geq 0}$ by \eqref{g0lt}. To get left-continuity of $(\ell_t)_{t>\Tgel}$, consider $t > \Tgel$, and let $\ell$ be the left limit of $\ell_s$ at $t$. We have $\ell \leq \ell_{t+(t-\Tgel)/2} < 1$, so by the continuity property (c1) above,
\[
1 = \phi_s(\ell_s) \underset{s \to t^-}{\to} \phi_t(\ell).
\]
Hence $\phi_t(\ell)=1$. Assume $\ell > \ell_t$ (that is, $\ell$ is the second point where $\phi_t$ reaches $1$). Take $x \in (\ell_t,\ell)$. By property (b5), $\phi_t(x) > 1$. But on the other hand, $x < \ell \leq \ell_s$ for $s <t$, so $\phi_s(x) \leq 1$, and so $\phi_t(x) \leq 1$, and this is a contradiction. So $\ell=\ell_t$ and $(\ell_t)_{t\geq 0}$ is indeed continuous. This concludes the proof of (3) and of the Lemma.
\end{enumerate}
\end{proof}

Finally, we will see that for $t > \Tgel$, $\ell_t=m_t$, so that $\phi_t$ increases from 0 to 1, which is its maximum, and then decreases. To this end, recall that $(\ell_t)_{t\geq 0}$ is monotone non-increasing and that $(\ell_t)_{t\geq 0}$ and $(\phi_t)_{t\geq 0}$ are continuous, so the chain rule for Stieltjes integrals and \eqref{phi} give
\begin{align*}
1 = \phi_t(\ell_t) & = \phi_0(\ell_0) + \int_0^t \phi_s'(\ell_s) \dl_s + \int_0^t \dplus{\phi_s}{s}(\ell_s) \ds \\
& = 1 + \int_0^t \phi_s'(\ell_s) \dl_s + \int_0^t \phi_s(\ell_s)(M_s-g_0(\ell_s)) \ds
\end{align*}
that is, with \eqref{g0lt},
\begin{equation}
\phi_t'(\ell_t) \dl_t = 0.
\label{phitprimlt}
\end{equation}
Hence,  $\dl_t$-a.e $\phi_t'(\ell_t)=0$, i.e. $\ell_t=m_t$. This is actually true for all $t > \Tgel$, as we shall now prove. This result also has its counterpart in the model with arms, namely part 3 of the proof of Theorem \ref{thsmolubras}.
\begin{lemma}\label{lemlt2}
For every $t > \Tgel$, $\phi_t'(\ell_t)=0$, i.e. $\ell_t=m_t$, the point where $\phi_t$ attains its maximum.
In particular,
\begin{equation}
\ell_t \, g_0'(\ell_t)=\frac1t, \qquad \forall \, t>\Tgel.
\label{deflt1}
\end{equation}
\end{lemma}
\begin{proof}
First, recall that $\phi_t$ is increasing on $[0,\ell_t]$, so that $\phi_t'(\ell_t) \geq 0$, that is
\begin{equation}
\ell_t g_0'(\ell_t) \leq \frac1t.
\label{ineg}
\end{equation}
Assume now that there is a $t > \Tgel$ such that $\phi_t'(\ell_t)>0$, and consider
\[
s = \sup \{ r \in (\Tgel,t) : \, \phi_r'(\ell_r) = 0 \}.
\]
As noted before, $t \mapsto \ell_t$ is strictly decreasing for $t > \Tgel$ for any $t > \Tgel$, so $\dl_t([\Tgel,\Tgel+\varepsilon[)>0$ for all $\varepsilon>0$. Hence there are points $r<t$ where $\phi_r'(\ell_r)=0$, and thus the definition of $s$ does make sense.

Take now $(r_n)$ a sequence of points such that $T < r_n < t$, $\phi_r'(\ell_r)=0$ and $(r_n)$ converges to $s$. Since $0 < \ell_s < 1$, by property (c2) above, we get
\[
0 = \phi'_{r_n} (\ell_{r_n}) \to \phi'_s(\ell_s)
\]
so that $\phi'_s(\ell_s)=0$. This shows that $s < t$, and that for $r\in(s,t)$, $\phi'_r(\ell_r) >0$. Hence, by continuity of $(\ell_r)_{r\geq 0}$ and by \eqref{phitprimlt}, $(\ell_r)_{r\in[s,t]}$ is constant. This gives
\[
\frac1s = \ell_s \, g_0'(\ell_s) = \ell_t \, g_0'(\ell_t) \leq \frac1t
\]
which is a contradiction since $s< t$. In particular, $\phi'_t(\ell_t)=0$ implies \eqref{deflt1}.
\end{proof}

\begin{proof}[Proof of Proposition \ref{uniuni}]
By Lemma \ref{lemlt2}, necessarily $M_t = M_0$ on $[0,\Tgel]$ and for $t> \Tgel$, $M_t := g_t(1) = g_0(\ell_t)$, where
\begin{equation}
\ell_tg_0'(\ell_t)=\frac1t.
\label{deflt}
\end{equation}
Since $x \mapsto x g_0'(x)$ is strictly increasing from $[0,1]$ to $[0,K]$, where $K=\la \mu_0 , m^2 \ra =1/\Tgel$, this equation has a unique solution for $t > \Tgel$. Hence $M_t$ is uniquely defined.
Therefore $\alpha_t$ and $\phi_t$ are uniquely determined by $g_0$,
so we can define $\phi_t$ as in \eqref{phit}, and
Lemma \ref{lemlt1} shows that $g_t(\phi_t(x))=g_0(x)$ for $x \in [0,\ell_t]$, and that $\phi_t$ is a bijection from $[0,\ell_t]$ to $[0,1]$. So it has a right inverse $h_t$, and compounding by $h_t$ in the previous formula gives
\begin{equation}
g_t(x)=g_0(h_t(x))
\label{gt}
\end{equation}
for all $x \in [0,1]$, $t \geq 0$. Thus $g_t$ can be expressed by a formula involving only $g_0$ and in particular, $(\mu_t)_{t\geq 0}$ depends only on $\mu_0$. This shows the uniqueness of a solution to Smoluchowski's equation \eqref{smolu2}.
\end{proof}

\subsection{Behavior of the moments}

In this paragraph, we will study the behavior of the first and second moment of $(\mu_t)_{t\geq 0}$ as time passes,
showing how to prove rigorously and recover the results of \cite{ErnstCP}. For more general 
coagulation rates, one can obtain upper bounds of the same nature, see \cite{LauCCC}.

First consider the mass $M_t = \la \mu_t, m \ra$. We will always assume that $\Tgel < + \infty$. Let us start with the following lemma.
\begin{lemma}
Let $\nu \in \Mcc$ be a measure which integrates $x \mapsto y^x$ for small enough $y > 0$. Let $m_0$ be the infimum of its support. Then
\[
\lim_{y \to 0^+} \frac{\la \nu, x y^x \ra}{\la \nu, y^x \ra} = m_0.
\]
\label{lemprob}
\end{lemma}
\begin{proof}
First, note that $x y^x \geq m y^x$ $\nu$-a.e. so
\[
\liminf_{y \to 0} \frac{\la \nu, x y^x \ra}{\la \nu, y^x \ra} \geq m_0.
\]
Let us prove now that
\[
\limsup_{y \to 0}  \frac{\la \nu, x y^x \ra}{\la \nu, y^x \ra} \leq m_0.
\]
Assume this is not true. Then, up to extraction of a subsequence, we may assume that there exists $\a > 0$ such that for arbitrary small $y \in (0,1)$, $\la \nu, x y^x \ra \geq (m_0+ \a) \la \nu, y^x \ra$. Hence $\la \nu, (x-m_0 - \a)y^x \ra \geq 0$, so
\begin{equation}
\la \nu, (x-m_0 - \a) y^x \un{\{x > m_0+\a\}} \ra \geq \la \nu, (m_0+\a-x) y^x \un{\{m_0 \leq x \leq m_0+\a\}} \ra.
\label{eqlem}
\end{equation}
But
\begin{align*}
\la \nu, (m_0+\a-x) y^x \un{\{m_0 \leq x \leq m_0+\a\}} \ra & \geq \la \nu, (m_0+\a-x) y^x \un{\{m_0 \leq x \leq m_0+\a/2\}} \ra \\
& \geq \la \nu, (m_0+\a-x) \un{\{m_0 \leq x \leq m_0+\a/2\}} \ra y^{m_0 + \a/2}
\end{align*}
and
\[
\la \nu, (x-m_0 - \a) y^x \un{\{x > m_0+\a\}} \ra \leq \la \nu, (x-m_0 - \a) \un{\{x > m_0+\a\}} \ra y^{m_0 + \a}.
\]
With \eqref{eqlem}, this shows that
\[
\la \nu, (x-m_0 - \a) \un{\{x > m_0+\a\}} \ra y^{\a/2} \geq \la \nu, (m_0+\a-x) \un{\{m_0 \leq x \leq m_0+\a/2\}} \ra
\]
and having $y$ tend to zero gives
\[
0 \geq \la \nu, (m_0+\a-x) \un{\{m_0 \leq x \leq m_0+\a/2\}} \ra
\]
which is a contradiction since $\nu([m_0,m_0+ \a/2]) > 0$.
\end{proof}
\begin{cor}
The mass of the system is continuous and positive. It is strictly decreasing on $[\Tgel,\pinf)$. Moreover, denote $m_0 = \inf \mathrm{supp}\ \mu_0$. Then
\[
\lim_{t \to \pinf}\frac{1}{tM_t} = m_0.
\]
\end{cor}
\begin{proof}
Recall that  $M_t=g_0(\ell_t)$ so the first properties follow from Lemma \ref{ti}. Denote now $\nu(\dm)=m\mu_0(\dm)$. For $t > \Tgel$, $\ell_tg_0'(\ell_t)=1/t$, so
\[
\frac{1}{tM_t}=\frac{\la \nu, x\ell_t^x \ra}{\la \nu, \ell_t^x \ra}
\]
and since $\ell_t \to 0$ when $t \to \pinf$, this tends to $m_0$ by Lemma \ref{lemprob}.
\end{proof}
We can also study the behavior of the mass for small times. Recall that before gelation, the mass is constant at $1$. We have seen that it is continuous at the gelation time. We may then wonder if its derivative is continuous, that is if $\dot{M}_{\Tgel+}$ is zero or not.
\begin{lemma}
The right derivative of $M$ at $\Tgel$ is given by
\[
\dot{M}_{\Tgel+}= - \lim_{x \to 1^-} \frac{g_0'(x)^3}{g_0'(x)+xg_0''(x)} \in [- \infty,0]
\]
provided the limit exists.
\label{Mpoint}
\end{lemma}
\begin{proof}[Proof of Lemma \ref{Mpoint}]
For $t > \Tgel$, $f(\ell_t)=1/t$ with $f(x)=xg_0'(x)$, and $0 < \ell_t < 1$. But $f'(\ell_t) \neq 0$, so by the inverse mapping theorem, $(\ell_t)_{t\geq 0}$ is differentiable and
\[
\dot{\ell}_t=- \frac{1}{t^2 f'(\ell_t)}.
\]
Using the fact that $M_t = g_0(\ell_t)$, it is then easy to see that
\[
\dot{M}_t=-\ell_t^2 \frac{g_0'(\ell_t)^3}{g_0'(\ell_t)+\ell_tg_0''(\ell_t)}.
\]
Since $(\ell_t)_{t\geq 0}$ is continuous at $\Tgel$ and $\ell_{\Tgel}=1$, the result follows.
\end{proof}
Recall that the gelation time is precisely the first time when the second moment $\la \mu_t, m^2 \ra$ of $(\mu_t)_{t\geq 0}$ becomes infinite. It actually remains infinite afterwards.
\begin{cor}
For all $t \geq \Tgel$, $\la \mu_t, m^2 \ra = \pinf$.
\end{cor}
\begin{proof}
Note that
\[
\la \mu_t , m^2 \ra = g_t'(1),
\]
this formula being understood as a monotone limit. By \eqref{gtphit}, for $x < \ell_t$
\[
\phi_t'(x) g_t'(\phi_t(x))=g_0'(x).
\]
When $x \to \ell_t^-$, $\phi_t'(x) \to 0$ by Lemma \ref{lemlt2}, and $g_0'(x) \to g_0'(\ell_t) \neq 0$ since $\ell_t > 0$. So
\[
g_t'(\phi_t(\ell_t))=g_t'(1)=\pinf.
\]
\end{proof}

\subsection{Existence of solutions}\label{exist}

Existence of solutions of \eqref{smolu2} is a well-known topic, see e.g. \cite{FouLauMLP}. However, the case $M_0=+\infty$ is apparently new, so that
we give a short proof for the general case based on previous papers, mainly \cite{ShirRoesSR}.

Let now $\mu_0\in\Mf$ be as in the statement
of Theorem \ref{thsmolu} and let us set $g_0$ as in \eqref{g_0}, $\ell_t$ and
$M_t$ as in point (1) of Proposition \ref{uniuni}, $\alpha_t$ and $\phi_t$ as in
\eqref{alphat} and \eqref{phit}. Then it is easy to see that $\phi_t$ admits a right
inverse $h_t$ satisfying \eqref{ht}, and we can thus define
\[
g_t(x):= g_0(h_t(x)), \qquad t\geq 0, \ x\in[0,1].
\]
It is an easy but tedious task to check that $g_t$ satisfies \eqref{PDE} and all properties
(a1)-(a4) above.
In particular,
if $g_0(1)=+\infty$ then $h_t(1)<1$ and therefore
$g_t(1)<+\infty$ for all $t>0$.
Following \cite{ShirRoesSR}, we can now prove the following.
\begin{prop}\label{shir}
For all $t>0$ there exists $\mu_t\in\Mf$ such that
\[
g_t(x) =\la \mu_t, mx^m \ra = \int_{(0,\pinf)} mx^m \: \mu_t(\dm),  \qquad x\in[0,1].
\]
\end{prop}
\begin{proof}
Let $t>0$ be fixed. We set for all $y\geq 0$
\[
\Phi(y) := g_0(e^{-y}), \quad
\Gamma(y) := tg_0(e^{-y}), \quad
G(y) := \Gamma(y)+y-\log\alpha_t= -\log\phi_t(e^{-y}).
\]
We recall
that $f:[0,+\infty)\mapsto[0,+\infty)$ is completely monotone if
$f$ is continuous on $[0,+\infty)$, infinitely many times differentiable on $(0,+\infty)$ and
\[
(-1)^k \frac{d^kf}{dy^k} (y) \geq 0,
\qquad \forall \, k\geq 0, \, y\in(0,+\infty).
\]
It is easy to see that $\Phi$ and $\Gamma$ are completely monotone. Moreover, $G$ has a right inverse
\[
G^{-1}:[0,+\infty)\mapsto[\log(1/\ell_t),+\infty), \quad G^{-1}(y)=-\log h_t(e^{-y}), \qquad y\geq 0,
\]
and therefore by the definitions
\[
g_0(h_t(e^{-y}))= \Phi(G^{-1}(y)), \qquad y\geq 0.
\]
By \cite[Thm. 3.2]{ShirRoesSR}, $\Phi\circ G^{-1}$ is completely monotone and therefore, by Bernstein's Theorem,
there exists a unique $\nu_t\in\Mf$ such that
\[
g_t(e^{-y})=g_0(h_t(e^{-y}))= \Phi(G^{-1}(y))=\int_{(0,\pinf)} e^{-ym}  \,  \nu_t(\dm), \qquad y\geq 0.
\]
Since $g_t(1)<+\infty$ for all $t>0$, we obtain that $\la \nu_t,m\ra<+\infty$, so that
we can set $\mu_t(\dm):=m\, \nu_t(\dm)$, and we have found that there is a unique $\mu_t\in\Mf$ such that
\[
g_t(x)=g_0(h_t(x))= \int_{(0,\pinf)} x^{m} \, m \,  \mu_t(\dm), \qquad x\in(0,1].
\]
\end{proof}

In order to show that $(\mu_t)_{t\geq 0}$ is a solution of Smoluchowski's equation in the sense
of Definition \ref{defsolsmolu}, we have to check
that $\int_0^\varepsilon M_t^2 \, \mathrm{d}t<+\infty$ for all $\varepsilon>0$. This is the content of the next result.

\begin{lemma}\label{gutt}
If $(\mu_t)_{t\geq 0}$ is the family constructed in Proposition \ref{shir}, then for all $\varepsilon>0$,
$\int_0^\varepsilon \la \mu_s, m\ra^2 \ds<  \pinf$.
\end{lemma}
\begin{proof}
If $M_0<+\infty$ then there is nothing to prove, since $(M_t)_{t\geq 0}$ is monotone non-increasing,
so let us consider the case $M_0=+\infty$ and thus $\Tgel=0$.
Since $M_t=g_0(\ell_t)$ is bounded and continuous for $t\in[\delta,\varepsilon]$ for all $\delta\in(0,\varepsilon)$,
we have by \eqref{ell} and \eqref{g_0}
\[
\begin{split}
\int_\delta^\varepsilon M_t^2\, \mathrm{d}t & =\int_\delta^\varepsilon g_0^2(\ell_t)\, \mathrm{d}t =
\varepsilon g_0^2(\ell_\varepsilon) -\delta g_0^2(\ell_\delta) - \int_\delta^\varepsilon 2tg_0(\ell_t)g_0'(\ell_t) \, \mathrm{d} \ell_t \\
& \leq \varepsilon g_0^2(\ell_\varepsilon) - \int_\delta^\varepsilon 2\, g_0(\ell_t)\, \frac{\mathrm{d} \ell_t}{\ell_t} =
\varepsilon g_0(\ell_\varepsilon) + 2\int_{\ell_\varepsilon}^{\ell_\delta} g_0(y)\, \frac{\mathrm{d} y}y \\
& \leq \varepsilon g_0(\ell_\varepsilon) +\frac{2}{\ell_\varepsilon} \, \la \mu_0, \frac m{1+m}\ra
 \leq \varepsilon g_0(\ell_\varepsilon) +\frac{2}{\ell_\varepsilon} \, \la \mu_0, m\wedge 1\ra.
\end{split}
\]
Letting $\delta\downarrow 0$, by \eqref{condition} we obtain the desired result.
\end{proof}

We now finish the proof of existence of a solution by showing that $(\mu_t)_{t\geq 0}$
indeed solves \eqref{smolu2}. By choosing $x=e^{-y}$, $y\geq 0$, in \eqref{PDE}, we find
an equality between Laplace transforms.  Since the Laplace transform is one-to-one,
then we obtain \eqref{smolu2}.
\begin{rem}
{\rm In the proof of uniqueness, we may only require that $\la \mu_0, m y^m \ra < \pinf$ for every $y \in [0,1)$. However, the same kind of computation as in Lemma \ref{gutt} shows that if this the case, but $\la \mu_0, m \wedge 1 \ra = + \infty$, then $\int_0^t M_s^2 \ds = + \infty$ for all $t > 0$, in contradiction with Definition \ref{defsolsmolu} of a solution.}
\end{rem}

\section{Flory's equation}

We will now consider the modified version of Smoluchowski's equation,
also known as {\it Flory's equation}, with a multiplicative kernel.
\begin{definition}\label{def2.1}
Let $\mu_0\in\Mcc$.
We say that a family $(\mu_t)_{t\geq 0}\subset \Mcc$ solves Flory's equation \eqref{smolu2} if
\begin{itemize}
\item for every $t>0$, $\int_0^t \la \mu_s(\dm), m\ra^2 \ds<  \pinf$,
\item for all $\phi \in \Cc$ and $t>0$
\begin{equation}\label{flory}
\begin{split}
\la \mu_t , \phi \ra & = \la \mu_0 , \phi \ra + \frac12 \int_0^t \la \mu_s(\dm) \mu_s(\dm'), mm' \Dphi \ra \ds \\
 & \quad - \int_0^t \la \mu_s, \phi \ra \la \mu_0(\dm) - \mu_s(\dm), m \ra \ds,
\end{split}
\end{equation}
\item if $\la \mu_0 , m^2 \ra < \pinf$, then $t \mapsto \la \mu_t , m^2 \ra$ is bounded in a right neighbor\-hood of 0.
\end{itemize}
\end{definition}
In equation \eqref{flory}, the mass that vanishes in the gel interacts with the other particles. It is a
modified Smoluchowski's equation, where a term has been added, representing
the interaction of the particles of mass $m$ with the gel, whose mass is
\[
\la \mu_0 - \mu_s , m \ra
\]
i.e. precisely the missing mass of the system. Notice that in this case the equation
makes sense only if $\la \mu_0 , m \ra<+\infty$.

The mass is expected to decrease faster in this case than for \eqref{smolu2}. This is actually true, as we can see in the following result.
\begin{theorem}\label{thflory}
Let $\mu_0\in\Mcc$ a non-null measure such that
$\la \mu_0 , m \ra <+\infty$,  and set
\[
M_0:=\la \mu_0 , m \ra \in (0,\pinf), \qquad
K := \la \mu_0 , m^2 \ra \in (0,\pinf].
\]
Let $\Tgel := 1/K \in [0,+\infty)$.
Then Flory's equation \eqref{flory}
has a unique solution $(\mu_t)_{t\geq 0}$ on $\R^+$. It has the following properties.
\begin{enumerate}
\item We have $M_t=g_0(l_t)$, where $l_t=1$ for $t\leq  \Tgel$ and, for $t> \Tgel$, $l_t$
is uniquely defined by
\[
l_t=e^{-t(M_0-g_0(l_t))}, \qquad l_t\in[0,1).
\]
Therefore $t\mapsto M_t$ is continuous on $[0,+\infty)$, constant on $[0,\Tgel]$, strictly decreasing on $[\Tgel,\pinf)$ and analytic on $\R^+ \backslash \{ \Tgel \}$.
\item The function $\phi_t(x) = xe^{t(M_0-g_0(x))}$ has a right inverse $h_t \, : \, [0,1] \to [0,l_t]$. The generating function $g_t$ of $(\mu_t)_{t\geq 0}$ is given for $t \geq 0$ by
\[
g_t(x)=g_0(h_t(x)).
\]
\item Let $m_0 = \inf \mathrm{supp}\ \mu_0\geq 0$. Then, when $t \to \pinf$,
\[
M_t e^{m_0t} \to m_0 \mu_0(\{m_0\})
\]
and for every $\eps > 0$
\[
M_t e^{(m_0+\eps)t} \to \pinf.
\]
\item The second moment $\la m^2,c_t \ra$ is finite on $\R^+ \backslash \{ \Tgel \}$ and infinite at $\Tgel$.
\end{enumerate}
\end{theorem}
\begin{rem}{\rm
\begin{itemize}
	\item Norris \cite[Thm 2.8]{NorrisCC} has a proof of global uniqueness of Flory's equation \eqref{flory} for slightly less general initial
conditions ($\mu_0$ such that $\la\mu_0,1+m\ra<+\infty$), but for a much more general model.
  \item When $m_0 > 0$, it was already observed (Proposition 5.3 in \cite{EscobedoGaMC}) that the mass decays (at least) exponentially fast (see also \cite{ErnstKoG,StellZiff,ZiffKoP}).
\end{itemize}
}
\end{rem}
\begin{proof}[Proof of Theorem \ref{thflory}] The proof is very similar to (and actually easier than) that of
Theorem \ref{thsmolu}.
\begin{enumerate}
\item Arguing as in the proof of Lemma \ref{lemM}, we obtain easily that $(M_t)_{t\geq 0}$ is
monotone non-increasing and right-continuous. As in Lemma \ref{lemmt}, if
$t \mapsto \la \mu_t , m^2 \ra$ is bounded on some interval $[0,T_0]$, then $M_t = M_0$ for $t \in [0,T_0]$ and
therefore $(\mu_t)_{t\geq 0}$ is a solution of Smoluchowski's equation \eqref{smolu2} on $[0,T_0]$.
\item Consider initial concentrations $\mu_0$ as in the statement, a solution $(\mu_t)_{t\geq 0}$ to Flory's equation and $g_t(x)$,
$x\in[0,1]$, generating function of $m\, \mu_t(dm)$. Then $g_t$ solves the PDE
\begin{equation}
\d{g_t}{t}=x(g_t-M_0)\d{g_t}{x}, \qquad \forall \, t>0, \, x\in[0,1],
\label{PDE2}
\end{equation}
the same as the one obtained for Smoluchowski's equation before gelation. It may be solved using the method of characteristics. Indeed, the mapping
\begin{equation}\label{phi2}
\phi_t(x)=x e^{t(M_0-g_0(x))}=x+\int_0^t (M_0-g_0(x))\, \phi_s(x) \ds,
\end{equation}
has the following properties
\begin{itemize}
\item[(d1)] $\phi_t(0)=0$, $\phi_t(1)=1$.
\item[(d2)] For all $t\geq 0$, $\phi_t'(x)=e^{t(M_0-g_0(x))}(1-txg'_0(x))$.
\item[(d3)] For $t \leq \Tgel$, $\phi_t(\cdot)$ is increasing; therefore,
$\phi_t(x)\in[0,1]$ for all $x\in[0,1]$ and $\phi_t(x)=1$ if and only if $x=1$
\item[(d4)] For $t  > \Tgel$, $\phi_t(\cdot)$ is increasing on $[0,m_t]$ and decreasing on $[m_t,1]$, where $m_t$
is the unique $x\in(0,1)$ such that $\phi'_t(x)=0$, i.e. such that $txg_0'(x)=1$.
\item[(d5)] For $t>\Tgel$, $\phi_t(m_t)>1$, since $\phi_t(1)=1$ and $\phi_t'(1)<0$.
Therefore there is a unique $l_t\in(0,m_t)$ such that $\phi_t(l_t)=1$.
\item[(d6)] For $t > \Tgel$, $\phi_t'(l_t) \neq 0$, since $l_t<m_t$.
\end{itemize}
\begin{figure}[htb]
\begin{tabular}{cc}
\includegraphics[width=7.1cm]{phitpregel.png}& \includegraphics[width=7.1cm]{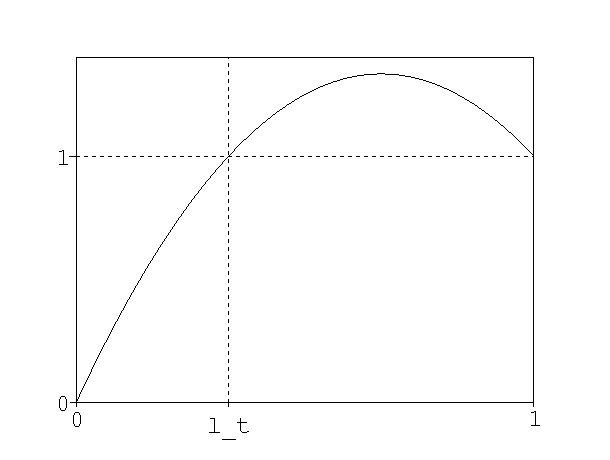}
\end{tabular}
\caption{$\phi_t$ before and after gelation.}
\end{figure}
Setting $l_t:=1$ for $t\leq \Tgel$,
$\phi_t$ is thus a continuous bijection from $[0,l_t]$ to $[0,1]$, with continuous inverse function $h_t:[0,1]\mapsto
[0,l_t]$. By using \eqref{PDE2} and
\eqref{phi2} and arguing as in part (i) and (ii) of the proof of Lemma \ref{ti}, we can see that the function $u_t(x):=g_t(\phi_t(x))-g_0(x)$
satisfies $u_t(x)=u_0(x)=0$ for all $t\geq 0$ and $x\in[0,l_t]$. Therefore
the only solution of the PDE \eqref{PDE2} is given by
\begin{equation}
g_t(x)=g_0(h_t(x)), \qquad t\geq 0, \, x\in[0,1].
\label{gt2}
\end{equation}
Flory's equation has thus a unique solution on $\R^+$, and its generating function is $g_t$.
\item We have seen in (d5) above that, for $t>\Tgel$, there is a unique
$l_t\in[0,1)$ such that $\phi_t(l_t)=1$. The relation $\phi_t(l_t)=1$ with $l_t\in[0,1)$ is equivalent
to $l_t=e^{-t(M_0-g_0(l_t))}$ with $l_t\in[0,1)$. This relation implies that $t\mapsto l_t$ is
analytic for $t>\Tgel$. A differentiation shows that
\[
\frac{dl_t}{dt} = -\frac{(M_0-g_0(l_t))l_t}{1-tg_0'(l_t)l_t}<0, \qquad t>\Tgel,
\]
since $g_0'(l_t)l_t<g_0'(m_t)m_t=1/t$ and $g_0(l_t)<g_0(1)=M_0$. Let $\ell$ be the
limit of $l_t$ as $t\downarrow \Tgel$: then we obtain $\ell=e^{-\Tgel(M_0-g_0(\ell))}$,
i.e. $\phi_{\Tgel}(\ell)=1$. By (d3) above, this is equivalent to $\ell=1$.
\item Since $M_t=g_t(1)=g_0(h_t(1))=g_0(l_t)$,
the properties of $t\mapsto M_t=g_0(l_t)$ follow from those of $t\mapsto l_t$.
Recall now that $\phi_t(l_t)=1$, that is
\begin{equation}
\log(l_t)=t(g_0(l_t)-1).
\label{lnlt}
\end{equation}
If the limit $l$ of $l_t$ as $t\to+\infty$ were nonzero, then passing to the limit in this equality would give $\log(l)=- \infty$. So $l=0$ and
\begin{equation}
\log l_t \sim -t.
\label{equiv}
\end{equation}
\begin{itemize}
\item Assume $m > 0$. Now, obviously $g_0(x) \leq x^m$, so
\[
\log(tg_0(l_t))=\log l_t + \log g_0(l_t) \leq \log t + m \log l_t \to - \infty.
\]
Hence $t g_0(l_t) \to 0$ and \eqref{lnlt} yields $\log l_t + t \to 0$. Hence
$l_t^m \sim e^{-mt}$.
Finally
\[
\lim_{t \to \pinf} M_t e^{mt} = \lim_{t \to \pinf} \frac{g_0(l_t)}{l_t^m} = m \mu_0(\{m\})
\]
since by dominated convergence, $g_0(x) x^{-m} \to m \mu_0(\{m\})$ when $x \to 0$. Now, by monotone convergence, if $m' > m$, then $g_0(x)x^{-m'} \to \pinf$ when $x$ tends to 0, whence
\[
\lim_{t \to \pinf} M_t e^{m't} = \lim_{t \to \pinf} \frac{g_0(l_t)}{l_t^{m'}} = \pinf.
\]
\item Assume now $m = 0$ and let $\eps > 0$. By monotone convergence
$g_0(x)x^{-\eps} \to \pinf$ as $x\downarrow 0$,
so using \eqref{equiv} we see that
$g(l_t)e^{-\eps t} \to \pinf$ as $t\uparrow + \infty$, which is the desired result.
\end{itemize}
\item Finally, \eqref{gt2} gives for $x <1$ and $t > \Tgel$
\[
g_t'(x)=g_0'(h_t(x))\, h_t'(x) = \frac{g_0'(h_t(x))}{\phi_t'(h_t(x))}.
\]
When $x \uparrow 1$, $h_t(x) \uparrow l_t < 1$, and $\phi_t'(h_t(x)) \to \phi_t'(l_t) \neq 0$
by (d6) above. So $\la \mu_t , m^2 \ra = g_t'(1) < \pinf$.
\item Existence of a solution of \eqref{flory} follows arguing as in section \ref{exist}.
\end{enumerate}
\end{proof}

\begin{cor}\label{leq}
Let $\mu_0\in\Mcc$ such that $\la\mu_0,m\ra<+\infty$ and let $(\mu_t^S)_{t\geq 0}$ and
$(\mu_t^F)_{t\geq 0}$ the solutions of \eqref{smolu2}, respectively, \eqref{flory}. Then
\begin{itemize}
\item $\mu_t^S\equiv \mu_t^F$ for all $t\leq \Tgel:=1/\la\mu_0,m^2\ra$;
\item $\la\mu_t^F,m\ra
<\la\mu_t^S,m\ra$ for all $t> \Tgel$.
\end{itemize}
\end{cor}
\begin{proof}
For all $t\leq \Tgel$, $\la\mu_t^F,m\ra=\la\mu_0^F,m\ra$ and therefore $\mu_t^F$ solves \eqref{smolu2},
so that by uniqueness of Smoluchowski's equation we have that $\mu_t^S= \mu_t^F$. For $t> \Tgel$
we have that $\la\mu_t^F,m\ra=g_0(l_t)$ while $\la\mu_t^S,m\ra=g_0(\ell_t)$, where $l_t$ and $\ell_t$
are defined respectively by
\[
l_t=e^{-t(M_0-g_0(l_t))}, \quad l_t\in[0,1)
\]
and
\[
\ell_tg_0'(\ell_t)=\frac1t.
\]
In points (d4) and (d5) of the proof of Theorem \ref{thflory}, we have shown that $l_t<m_t$ where
$tm_tg_0'(m_t)=1$, so that $m_t=\ell_t < l_t$. Hence $\la \mu_t^F, m \ra = g_0(\ell_t) < g_0(l_t) = \la \mu_t^S, m \ra$.
\end{proof}

As anticipated, the mass decreases faster in Flory's case than for Smoluchowski's equation. In particular,  in
Flory's case $\la \mu_t , m^2 \ra$ becomes finite immediately after gelation, the mass remaining however continuous (we can think that the big particles, which have the biggest influence on this second moment, disappear into the gel). Moreover, if
$\inf \mathrm{supp}\ \mu_0>0$ then the mass decays exponentially fast, which is to be compared with the slow decrease in $1/t$ in Smoluchowski's equation.

\begin{rem}
{\rm The mass in Flory's equation may decrease slower if $\inf \mathrm{supp}\ \mu_0$ $=0$. For instance, if $\mu_0(\dm)=e^{-m} \dm$, then $M_t \sim t^{-2}$.}
\end{rem}

\section{The model with limited aggregation}

We now turn to our main interest, namely Equation \eqref{smolubras0}. We apply the same techniques as above in a slightly more
complicated setting. After giving all details in Smoluchowski's case, we will give a shorter proof and focus on the differences with the proof of Theorem \ref{thsmolu}. As above, we can transform the system \eqref{smolubras0} into a non-local PDE problem, which we are able to solve, thus obtaining existence and uniqueness to \eqref{smolubras0}. More
precisely, we consider the following system.
\begin{definition}\label{def3.1}
Let $c_0(a,m)\geq 0$, $a\in\N$, $m\in\N^*$. We say that a family $(c_t(a,m))$, $t\geq 0$, $a\in\N$, $m\in\N^*$, is a solution of Smoluchowski's equation \eqref{smolubras} if
\begin{itemize}
\item for every $t>0$, $\int_0^t \la c_s, a\ra^2 \ds<  \pinf$,
\item for all $a\in\N$, $m\in\N^*$ and $t>0$,
\begin{equation} \label{smolubras}
\begin{split}
c_t(a,m)& = \ c_0(a,m)+ \\
& \quad + \int_0^t\frac12 \sum_{a'=1}^{a+1} \sum_{m'=1}^{m-1} a'(a+2-a') c_s(a',m')c_s(a+2-a',m-m') \ds \\
& \quad - \int_0^t \sum_{a' \geq 1} \sum_{m'\geq 1} aa'c_s(a,m) c_s(a',m') \ds,
\end{split}
\end{equation}
\item if $\la c_0 , a^2 \ra < \pinf$, then $t \mapsto \la c_t , a^2 \ra$ is bounded in a right neighborhood of 0.
\end{itemize}
\end{definition}
Because of the interpretation of $a$ as a variable counting the number of arms a particle possesses, it is more
natural to state \eqref{smolubras} in the discrete setting, as in \cite{BertoinTSS}. In particular, since at each coagulation
two arms are removed from the system, a non-integer initial number of arms would lead to an ill-defined dynamics.
One could however with no difficulty consider an initial distribution of masses on $(0,+\infty)$.

It is easy to see that $(c_t)$ is a solution to this equation if and only if the function
\begin{equation}
k_t(x,y):= \sum_{a=1}^{+\infty} \sum_{m=1}^{+\infty} a\, c_t(a,m) \, x^{a-1}\, y^m,
\label{kt}
\end{equation}
defined for $t\geq 0$, $y \in [0,1]$ and $x \in [0,1)$, satisfies
\begin{equation}\label{PDE3}
\left \{
\begin{array}{l}
\displaystyle k_t(x,y)= k_0(x,y)+\int_0^t\left[
\left( k_s(x,y) - xA_s \right) \d{k_s}{x}(x,y) - A_sk_s(x,y)\right] \ds, \\
A_t:=k_t(1,1)=\la c_t,a\ra.
\end{array}
\right.
\end{equation}
We may solve this PDE with the same techniques as above and obtain the following result.
\begin{theorem}\label{thsmolubras}
Consider initial concentrations $c_0(a,m)\geq 0$, $a\in\N$, $m\in\N^*$ such that $\la c_0, 1 \ra <+\infty$,
$A_0:=\la c_0, a \ra \in(0,+\infty]$
and with $K:=\la c_0 , a(a-1) \ra \in [0,\pinf]$. Then $K=+\infty$ whenever $A_0=+\infty$. Let
\begin{equation}
\Tgel= \begin{cases}
\frac{1}{K-A_0} \quad & {\rm if} \ A_0<K<+\infty, \\ 0 \quad & {\rm if} \ K=+\infty,
\\ +\infty \quad & {\rm if} \ K\leq A_0<+\infty.
\end{cases}
\label{Tgelbras}
\end{equation}
Then equation \eqref{smolubras} has a unique solution defined on $\R^+$.
When $\Tgel < \pinf$, this solution enjoys the following properties.
\begin{enumerate}
\item The number of arms $A_t := \la c_t , a \ra$ is continuous, strictly decreasing, and for all
$t>0$
\begin{equation}\label{to0}
A_t \leq \frac{A_0}{1+tA_0}\quad {\rm if} \ A_0<+\infty, \qquad A_t \leq \frac1t \quad {\rm if} \ A_0=+\infty.
\end{equation}
If we set
\[
\a_t = \exp\left(\int_0^t A_s \ds\right),
\]
then $\a_t$ is given by
\[
\a_t = 1+A_0t \quad \mathrm{for}\ t < \Tgel
\]
and for $t\geq \Tgel$
\begin{equation}
\a_t = \begin{cases}\G^{-1}(1+A_0\Tgel+t-\Tgel)\quad & {\rm if} \ A_0<+\infty, \\
\G^{-1}(1+t) \quad & {\rm if} \ A_0=+\infty,\end{cases}
\label{alphat2}
\end{equation}
where
\[
\G(x) =1+A_0\Tgel+\int_{1+A_0\Tgel}^x \frac{dr}{k_0(H(1/r))}, \quad x\geq1+A_0\Tgel,
\]
and $H:[G(0),G(1))\mapsto[0,1)$ is the right inverse of the increasing function
\begin{equation}\label{G}
G:[0,1)\mapsto[G(0),G(1)), \qquad
G(x) := x - \frac{k_0(x,1)}{k_0'(x,1)}, \quad x\in[0,1),
\end{equation}
with $G(0):=G(0^+)\leq 0$, and
\[
0 < G(1):=G(1^-)=\begin{cases} \displaystyle 1-\frac{A_0}K\quad & {\rm if} \ A_0<+\infty \\ 1 \quad & {\rm if} \ A_0=+\infty. \end{cases}
\]
\item Let $k_0$ be defined as in \eqref{kt},
and
\begin{equation} \label{Atbt}
A_t = \la c_t , a \ra, \qquad \a_t = \exp\left(\int_0^t A_s \ds\right),
\qquad \b_t = \int_0^t \frac{1}{\a_s^2} \ds.
\end{equation}
Consider
\[
\phi_t(x,y) := \a_t(x - \b_t k_0(x,y)), \qquad t\geq 0, \, x,y\in[0,1].
\]
Then
\begin{itemize}
\item $\phi_t(\cdot,1)$ attains its maximum at a point $\ell_t$ such that $\phi_t(\ell_t,1)=1$. For $t \leq \Tgel$,
$\ell_t =1$, and for $t > \Tgel$, $0 < \ell_t < 1$ and
\begin{equation}
\d{\phi_t}{x}(\ell_t,1) = 0.
\label{phitprimlt2}
\end{equation}
In particular, for $t > \Tgel$, $\ell_t$ is given by
\begin{equation}\label{phitprimlt22}
\ell_t = H \left ( \frac{1}{\a_t} \right ),
\end{equation}
where $H$ is the right inverse of the function $G$ defined above.
\item For every $y \in [0,1]$, $\phi_t(\cdot,y)$ has a right inverse $h_t(\cdot,y):[0,1]\mapsto[0,1]$.
\end{itemize}
\item The generating function $k_t$ defined by \eqref{kt} is given by
\begin{equation}
k_t(x,y) = \frac{1}{\a_t} k_0(h_t(x,y),y)
\label{ktxy}
\end{equation}
for $y \in [0,1]$, $x \in [0,1]$. In particular, for $t>0$
\begin{equation}\label{k0lt}
\a_t A_t = \a_t k_t(1,1) = k_0(\ell_t,1), \qquad
A_t =\frac{k_0(\ell_t,1)}{1+\int_0^tk_0(\ell_s,1) \ds}.
\end{equation}
\item The second moment $\la c_t, a^2 \ra$ is finite on $[0,\Tgel)$, infinite on $[\Tgel,+ \infty)$.
\end{enumerate}
\end{theorem}

\subsection{Proof}

The only major difference with respect to the proof of Theorem \ref{thsmolu} is the
additional variable $y$ in the generating function $k_t(x,y)$. However, the variable
$y$ plays the role of a parameter in the PDE \eqref{PDE3}, and this allows to adapt
all above techniques.

\begin{proof}[Proof of Theorem \ref{thsmolubras}] The case $K\leq A_0<+\infty$, for which $\Tgel= \pinf$
has already been treated in \cite[Thm. 2]{BertoinTSS}, so
that we can restrict here to the cases where $\Tgel<+\infty$.
When $\Tgel > 0$, Thm. 2 in \cite{BertoinTSS} also shows that $\a_t = 1 + A_0 t$ on $[0,\Tgel)$ (this however also requires that $\la a^2 ,c_t \ra$ be bounded in a neighborhood of 0: see point 3 of the proof of Lemma \ref{ti}).

\begin{enumerate}
\item First, by setting $u_t(x,y):=\alpha_tk_t(\phi_t(x,y),y)-k_0(x,y)$, we can see,
arguing as in points (i)-(ii) of the proof of Lemma \ref{ti}, that for all
$y\in(0,1]$ and $t>0$ there exists $\ell_t^0(y) < \ell_t(y)\in(0,1]$ such that
\begin{equation}\label{ktphit}
\alpha_tk_t(\phi_t(x,y),y)=k_0(x,y), \qquad \forall \, t\geq 0, \, y\in(0,1], \, x\in[\ell_t^0(y),\ell_t(y)]
\end{equation}
and $\phi_t(\cdot,y):[\ell_t^0(y),\ell_t(y)]\mapsto[0,1]$ is a continuous bijection and
has a continuous right inverse $h_t(\cdot,y):[0,1]\mapsto[\ell_t^0(y),\ell_t(y)]$.
\begin{figure}[htb]
\begin{tabular}{cc}
\includegraphics[width=7.1cm]{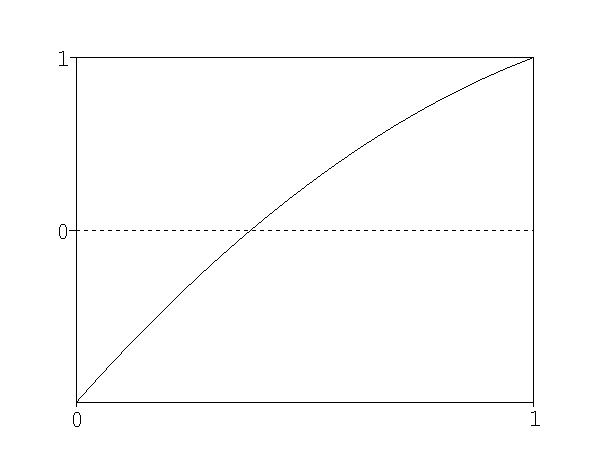}& \includegraphics[width=7.1cm]{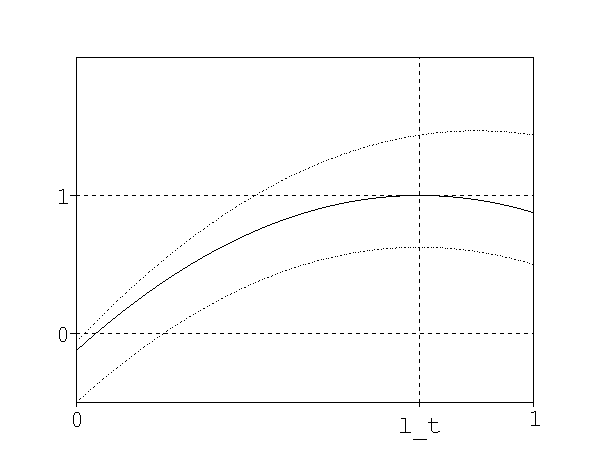}
\end{tabular}
\caption{$\phi_t(\cdot,1)$ before and after gelation. The dotted lines represent what $\phi_t$ may look like. The solid one is the actual $\phi_t$.}
\end{figure}
\item We denote for simplicity
\[
k_t(x):=k_t(x,1), \qquad \phi_t(x):=\phi_t(x,1), \qquad t\geq 0, \, x\in[0,1].
\]
For $y=1$, we set $\ell_t(1)=\ell_t$, i.e.
\[
1=\phi_t(\ell_t)=\alpha_t(\ell_t-\beta_tk_0(\ell_t)), \qquad t\geq 0.
\]
Arguing as in points (iv)-(v) of the proof of Lemma \ref{ti}, we can see that
$\ell_t=1$ for all $t\leq \Tgel$ and $\ell_t<1$ for all $t>\Tgel$. Moreover,
$t\mapsto \ell_t$ is continuous and monotone non-increasing.
Since $\phi_t$ is increasing on $[0,\ell_t]$, $\phi_t'(\ell_t)\geq 0$, i.e.
\[
\beta_t \leq \frac1{k_0'(\ell_t)},
\]
so that
\begin{equation}\label{4}
1 =\alpha_t(\ell_t-\beta_tk_0(\ell_t))\geq \alpha_t\, G(\ell_t),
\end{equation}
where we set $G(x):=x-\frac{k_0(x)}{k_0'(x)}$, $x\in[0,1)$. Notice that
\[
G'(x)= 1-\frac{(k_0'(x))^2-k_0(x)k_0''(x)}{(k_0'(x))^2}=\frac{k_0(x)k_0''(x)}{(k_0'(x))^2}>0,
\]
since $k_0$ is strictly convex (there is no gelation whenever $k_0'' \equiv 0$). Moreover $G(0)\leq 0$ and
\[
G(1)=1-\frac{A_0}K\quad {\rm if} \ A_0<+\infty, \qquad G(1)=1 \quad {\rm if} \ A_0=+\infty.
\]
Indeed, $k_0'(1)=K=\la c_0,a(a-1)\ra$ and, if $k_0(1)=A_0=+\infty$, then
\[
\lim_{x\uparrow 1} \frac{k_0(x)}{k_0'(x)} =0
\]
since, if $\liminf_{x\uparrow 1} \frac{k_0(x)}{k_0'(x)} > \varepsilon>0$, then
$k_0(1)\leq k_0(1-\delta)e^{\delta/\varepsilon}<+\infty$, for
some $\delta>0$, contradicting $k_0(1)=+\infty$.
In any case, $G$ has an inverse $H$, and $H(1/x)$ is defined for $x \in [1+A_0 \Tgel,+ \infty)$.
\item
Computing \eqref{ktphit} at $(x,y)=(\ell_t,1)$ we obtain
\begin{equation}\label{2}
k_0(\ell_t)=\alpha_tk_t(1)=\alpha_tA_t=\dplusdt{\alpha_t}.
\end{equation}
Let us notice that
\[
\phi_t(x) = x+\int_0^t \left(A_s \phi_s(x)-\frac{k_0(x)}{\alpha_s}\right) \ds.
\]
Then by \eqref{2}, analogously to \eqref{phitprimlt} above,
\[
0 = \mathrm{d} \, \phi_t(\ell_t) = \left(A_t\phi_t(\ell_t)-\frac{k_0(\ell_t)}{\alpha_t}\right)\mathrm{d}t
+ \phi'_t(\ell_t)\,\dl_t= \phi'_t(\ell_t)\,\dl_t.
\]
In particular, for $\dl_t$-a.e. $t$, $\phi'_t(\ell_t)=0$, i.e. $\beta_t = 1/k_0'(\ell_t)$,
and therefore
\[
1 =\alpha_t(\ell_t-\beta_tk_0(\ell_t))= \alpha_t\, G(\ell_t), \qquad \dl_t-{\rm a.e.} \, t.
\]
Then, by \eqref{4}, we can write (note that $H$ is well-defined on the considered interval)
\[
\ell_t \leq H\left(\frac1{\alpha_t}\right), \quad \forall \, t>\Tgel,
\qquad \ell_t = H\left(\frac1{\alpha_t}\right), \quad \dl_t-{\rm a.e.} \, t.
\]
Now,  by \eqref{2}, setting $\Lambda:\,]1+A_0\Tgel,+\infty[\,\mapsto\,]0,1[$, $\Lambda(z):=k_0\left(H\left(\frac1z\right)\right)$,
\[
\dplusdt{\alpha_t} \leq \Lambda(\alpha_t), \quad \forall \, t> \Tgel,
\qquad \dplusdt{\alpha_t} = \Lambda(\alpha_t), \quad \dl_t-{\rm a.e.} \, t.
\]
Since $\alpha_t>1+A_0 \Tgel$ for any $t>\Tgel$, we obtain that $k_0(\ell_t)\leq\Lambda(\alpha_t)<1$ for
all $t>\Tgel$. In particular, $\dl_t$ is not identically equal to $0$. Suppose that for some
$t>\Tgel$ we have $\phi_t'(\ell_t)>0$. We set
\[
s:=\sup\{r<t: \phi_r'(\ell_r)=0\}=\max\{r<t: \phi_r'(\ell_r)=0\}.
\]
Then for all $r\in\,]s,t[$ we must have $\phi_r'(\ell_r)>0$. Then for all
$r\in\,]s,t[$ we have $\ell_r=\ell_s$. But, by definition of $\beta$,
\[
\beta_r>\beta_s=\frac1{k_0'(\ell_s)}=\frac1{k_0'(\ell_r)}
\]
and this is a contradiction.
Therefore for all $t>\Tgel,$ we have $\dot{\alpha}_t = \Lambda(\alpha_t)$ for all
$t>\Tgel$ and the only solution of this equation with $\a_{\Tgel}=1+A_0 \Tgel$ is
given by \eqref{alphat2}.
\item In order to prove \eqref{k0lt}, let us note that by the preceding results
\[
\ddt{\a_t}=\a_t A_t = \a_t k_t(1,1) = k_0(\ell_t,1),
\]
\[
A_t = \ddt{}\log\a_t=\ddt{}\log\left(1+\int_0^tk_0(\ell_s,1) \ds\right)
=\frac{k_0(\ell_t,1)}{1+\int_0^tk_0(\ell_s,1) \ds}.
\]
\end{enumerate}
The rest of the proof follows the same line as that of Theorem \ref{thsmolu}.
\end{proof}

\section{The modified version}

Let us finally consider Flory's version of the model with arms. As in the case of Flory's equation
\eqref{flory}, we can consider only initial concentrations $c_0$ such that $A_0=\la c_0,a\ra<+\infty$.
Then, the equation we are interested in is
\begin{equation}
\begin{split}
\dt c_t(a,m) & = \frac12 \sum_{a'=1}^{a+1} \sum_{m'=1}^{m-1} a'(a+2-a') c_t(a',m')c_t(a+2-a',m-m') \\
 & \quad - \sum_{a' \geq 1} \sum_{m'\geq 1} aa'c_t(a,m) c_t(a',m') \\
 & \quad - \left ( \frac{A_0}{1+tA_0} - \sum_{a',m' \geq 1} a' c_t(a',m') \right ) a c_t(a,m).
\end{split}
\label{florybras}
\end{equation}
With the same techniques as above, we can prove the following result.
\begin{theorem}\label{thflorybras}
Consider initial concentrations $c_0(a,m)\geq 0$, $a\in\N$, $m\in\N^*$ such that $A_0:=\la c_0, a \ra \in(0,+\infty)$
and with $K:=\la c_0 , a(a-1) \ra \in [0,\pinf]$. Let $\Tgel$ be defined as in \eqref{Tgelbras}.
Then equation \eqref{florybras} has a unique solution defined on $\R^+$.
When $\Tgel < \pinf$, this solution enjoys the following properties.
\begin{enumerate}
\item We have
\begin{equation}
A_t=\frac{1}{1+tA_0} k_0(l_t)
\label{Atfl}
\end{equation}
where $l_t=1$ for $t\leq  \Tgel$ and, for $t> \Tgel$, $l_t$ is uniquely defined by
\[
l_t = \frac{t}{1+tA_0} k_0(l_t), \qquad l_t\in[0,1).
\]
Therefore $t\mapsto A_t$ is continuous and strictly decreasing on $[0,+\infty)$ and analytic on $\R^+ \bsl \{ \Tgel \}$.
\item The function $\phi_t(x,y)=(1+tA_0)x - t k_0(x,y)$ has, for every $y \in [0,1]$, a right inverse $h_t(\cdot,y) \, : \, [0,1] \to [0,l_t]$. The generating function $k_t$ defined in \eqref{kt} is given for $t \geq 0$ by
\begin{equation}
k_t(x,y) = \frac{1}{1+t A_0} k_0(h_t(x,y),y).
\label{ktfl}
\end{equation}
\item The second moment $\la a^2,c_t \ra$ is finite on $\R^+ \bsl \{ \Tgel \}$ and infinite at $\Tgel$.
\end{enumerate}
\end{theorem}
\begin{proof}
The proof follows the same line of reasoning as the one of Theorem \ref{thflory}. First, for every $y \in [0,1]$, $\phi_t(\cdot,y)$, as defined in the statement, has the following properties:
\begin{itemize}
\item[(i)] $\phi_t(0,y) \leq 0$, $\phi_t(1,y) \geq \phi_t(1,1) = 1$;
\item[(ii)] For $t \leq \Tgel$, $\phi_t(\cdot,y)$ is increasing , and in particular, there are unique $0 \leq l_t^0(y)<l_t(y) \leq 1$ such that $\phi_t(l_t^0(y),y)=0$ and $\phi_t(l_t(y),y) =1$;
\item[(iii)] For $t > \Tgel$, $\phi_t(\cdot,y)$ is increasing then decreasing for, and in particular, there are unique $0 \leq l_t^0(y) < l_t(y) < 1$ such that $\phi_t(l_t^0(y),y)=0$ and $\phi_t(l_t(y),y) =1$.
\end{itemize}
\begin{figure}[htb]
\begin{tabular}{cc}
\includegraphics[width=7.1cm]{phitpregelbras.png}& \includegraphics[width=7.1cm]{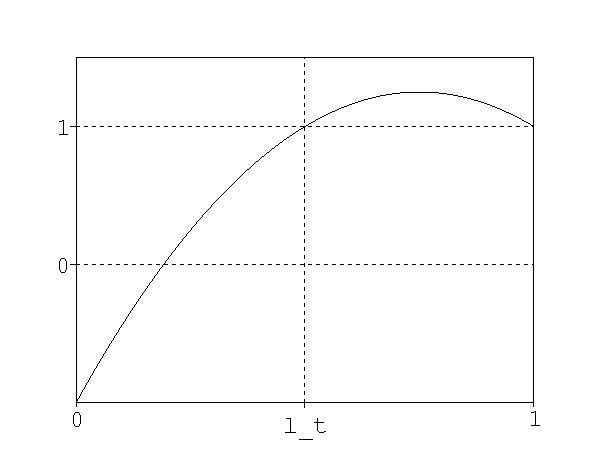}
\end{tabular}
\caption{$\phi_t(\cdot,1)$ before and after gelation.}
\end{figure}
In any case, it is easy to check that for $x \in [l_t^0(y),l_t(y)]$,
\[
\exp \left ( \int_0^t A_s \ds \right ) k_t(\phi_t(x,y),y)=k_0(x,y)
\]
where $A_t$ is defined by \eqref{Atfl}. Then, the properties above show that $\phi_t(\cdot,y)$ has a right inverse $h_t$ defined on $[0,1]$, and compounding by $h_t$ in the previous equation shows that \eqref{ktfl} holds. The other properties then follow easily.
\end{proof}

\section{Limiting concentrations}

We compute here some explicit formulas for the concentrations and their limit for the two models above. In the standard Smoluchowski and Flory cases, particles keep coagulating, and they all eventually disappear into the gel: $c_t(m) \to 0$ for every $m \geq 1$. When the aggregations are limited, there may remain some particles in the system, since whenever a particle with no arms is created, it becomes inert, and so it will remain in the medium forever. In the following, we consider monodisperse initial conditions, i.e. $c_0(a,m) = \mu(a) \un{\{m=1\}}$ for a measure $\mu$ on $\N$. We also denote
\[
\nu(m) = (m+1) \mu(m+1).
\]
In \cite{BertoinTSS}, it is assumed that $\nu$ is a probability measure, what we do not require. The results of \cite{BertoinTSS} can hence be recovered by taking $A_0=1$ below. Now, note the two following facts.
\begin{itemize}
	\item Equations \eqref{to0} and \eqref{Atfl} readily show that
	\begin{equation} \label{cinf}
	\cinf(a,m) := \lim_{t \to \pinf} c_t(a,m) = 0, \quad a \geq 1,
	\end{equation}
	that is, only particles with no arms remain in the medium (else, a coagulation ``should'' occur).
	\item There is an arbitrary concentration of particles with no arms at time 0, and they are the only particles with no arms and mass $1$ which will still be in the medium in the final state. Hence, the limit concentrations $\cinf(0,1)=c_0(0,1)$ have no physical meaning. We will thus only consider $\cinf(0,m)$ for $m \geq 2$.
\end{itemize}
Note now that if at time 0, each particle has zero or more than two arms, then obviously, this property still holds for any positive time. Rigorously, this is easy to check with the representation formula \eqref{ktxy} or \eqref{ktfl}. Then, because of \eqref{cinf},
\[
\cinf(m)=0
\]
for each $m \geq 2$. We thus rule out this trivial case by assuming that
\begin{equation} \label{hypnu}
\nu(0) > 0.
\end{equation}
This is actually a technical assumption which is needed to apply Lagrange's inversion formula in the proof of the following corollaries. We will relate our results to a population model known as the Galton-Watson process. For some basics on this topic, see e.g. the classic book \cite{AtNey}. The formula providing the total progeny of these processes was first obtained by Dwass in \cite{Dwass}.

\subsection{Modified model} \label{seccinffl}

\begin{cor} \label{corcinffl}
Let $c_t(a,m)$ be the solution to Flory's equation with arms \eqref{florybras} and with initial conditions $c_0(a,m) = \mu(a) \un{\{m=1\}}$ with $\mu(1) > 0$.
\begin{itemize}
	\item For all $t \geq 0$, $m \geq 2$, $a \geq 0$,
	\[
	c_t(a,m) = \frac{(a+m-2)!}{a!m!} \frac{t^{m-1}}{(1+tA_0)^{a+m-1}} \nu^{*m}(a+m-2).
	\]
	\item In particular, there are limiting concentrations $\cinf(a,m) = \cinf(m) \un{\{a=0\}}$ with
	\begin{equation} \label{cinffl}
	\cinf(m)=\frac{1}{m(m-1)} \nu^{*m}(m-2).
	\end{equation}
\end{itemize}
\end{cor}
\begin{proof}
With the notation of Theorem \ref{thflorybras}, we have
\[
(1+t A_0)h_t(x,y) - t y k_0(h_t(x,y)) = x, \quad k_t(x,y) = \frac{1}{1+t A_0} y k_0(h_t(x,y)).
\]
Up to some obvious changes (just replace $1+t$ by $1+tA_0$), these are precisely the equations solved in Section 3.2 of \cite{BertoinTSS} under the assumption \eqref{hypnu}. Theorem 2 and Corollary 2 therein hence give the desired result (with only $1+t$ replaced by $1+t A_0$).
\end{proof}

If $A_0=1$, which we may always assume up to a time-change, we observe as in \cite{BertoinTSS} that $2 (m-1) \cinf(0,m)$ is the probability for a Galton-Watson process with reproduction law $\nu$, started from two ancestors, to have total progeny $m$. This Galton-Watson process is (sub)critical when $K:=\sum_{a\geq 1}a(a-1)\mu(a) \leq 1$, that is, by Theorem \ref{thflorybras}, when there is no gelation, and supercritical when $K > 1$. Denote by $p_{\nu}$ its extinction probability, i.e. the smallest root of $k_0(x)=x$, so $p_{\nu} = 1$ when $K \leq 1$ and $p_{\nu} < 1$ when $K > 1$. Let us compute the mass at infinity, as in \cite{BertoinTSS}, by writing
\begin{align*}
M_{\infty} := \sum_{m \geq 1} m \cinf(m) & = \cinf(1) + \sum_{m \geq 2} \frac{1}{m-1} \nu^{*m}(m-2) \\
& = \cinf(1) + \sum_{a \geq 0} \nu(a) \sum_{m \geq a + 2} \frac{1}{m-1} \nu^{*m-1}(m-2-a) \\
& = \cinf(1) + \sum_{a \geq 0} \nu(a) \sum_{n \geq a + 1} \frac{1}{n} \nu^{*n}(n-1-a).
\end{align*}
Now, the Lagrange inversion formula \cite{Genefunc} shows that
\[
\frac{a+1}{n}\nu^{*n}(n-1-a)
\]
is precisely the coefficient of $x^n$ in the analytic expansion of $\phi(x)$ around $0$, where $\phi$ is the unique solution to $\phi(x) = x k(\phi(x))$. Hence 
\[
\sum_{n \geq a + 1} \frac{1}{n} \nu^{*n}(n-1-a) = p_{\nu},
\]
where $p_{\nu}$ is defined above. Note also that $\cinf(1) = \mu(0)$, so finally
\begin{equation}\label{Minfty}
M_{\infty} = \cinf(1) + \sum_{a \geq 0} \nu(a) \, \frac{1}{a+1} \,  p_{\nu}^{a+1}  = \sum_{a \geq 0} \mu(a) \,  p_{\nu}^a.
\end{equation}
The mass at time 0 is $M_0 = \sum \mu(a)$, so when there is no gelation, $p_{\nu}=1$ and no mass is lost in the gel. When there is gelation, $p_{\nu} < 1$ and the mass $M_0 - M_{\infty} > 0$ is lost in the gel. By Dwass' formula \cite{Dwass}, $M_{\infty}$ is also the probability that a Galton-Watson process, with reproduction law $\mu$ for the ancestor and $\nu$ for the others, has a finite progeny.

\subsection{Non-modified model} \label{seccinfsm}

\begin{cor}
Let $c_t(a,m)$ be the solution to Smoluchowski's equation with arms \eqref{smolubras} and with initial conditions $c_0(a,m) = \mu(a) \un{\{m=1\}}$ with $\mu(1) > 0$.
\begin{itemize}
	\item For all $t \geq 0$, $m \geq 2$, $a \geq 0$,
	\[
	c_t(a,m) = \frac{(a+m-2)!}{a!m!} \frac{\beta_t^{m-1}}{\alpha_t^a} \nu^{*m}(a+m-2)
	\]
	where $\a_t$ and $\b_t$ are defined in Theorem \ref{thsmolubras}.
	\item In particular, there are limiting concentrations $\cinf(a,m) = \cinf(m) \un{\{a=0\}}$ with
	\begin{equation}\label{cinfflb}
	\cinf(m)=\frac{1}{m(m-1)} \binf^{m-1} \nu^{*m}(m-2)
	\end{equation}
	where $\binf$ is defined by
	\[
	\binf = \frac{1}{k_0'(c)} = \frac{c}{k_0(c)}
	\]
	and $c$ is the unique solution to $k_0'(c) = k_0(c)/c$. Moreover, $\binf = 1$ when there is no gelation, and $\binf > 1$ otherwise.
\end{itemize}
\end{cor}
\begin{proof}
As for Corollary \ref{corcinffl}, the proof of the formula for $c_t(a,m)$ is the same as in \cite[Section 3.2]{BertoinTSS}, just replacing $1+t A_0$ by $\a_t$ and $t$ by $\a_t \b_t$. So we just have to find the limit of $\b_t$. First \eqref{alphat2} shows that $\a_t \to \pinf$,
hence, by  \eqref{phitprimlt22}, $\ell_t \to \ell_{\infty} = H(0)$. Now, \eqref{phitprimlt2} gives $\b_t = 1/k_0'(\ell_t)$, so $\b_t$ tends to
\[
\b_{\infty} = \frac{1}{k_0'(H(0))}
\]
where by definition $c := H(0)$ is the unique solution to $k_0'(c)=k_0(c)/c$. Finally, when there is gelation, $\a_t < 1+t$ after gelation because of \eqref{alphat2}, so by \eqref{Atbt}, $\binf > 1$.
\end{proof}

By a similar computation as above, we may also compute the mass at infinity in this case and get
\[
M_{\infty} = \sum_{a \geq 0} \mu(a) \,  c^a
\]
where $c$ is defined in the corollary. Note that $c$ is the slope of the straight line passing by $0$ and tangent to the graph of $k$, so $c > p_{\nu}$. In particular, less mass is lost than in Flory's case.

A final remark is that despite the striking resemblance between Formulas \eqref{cinfflb} and \eqref{cinffl}, the meaning of the factor $\binf$ is unclear. A probabilistic interpretation using the configuration model may explain its appearance.

\medskip

\noindent \textbf{Acknowledgements} We thank Jean Bertoin for useful discussions and advice.


\begin{thebibliography}{30}

\bibitem{AldousReview}
D.~J. Aldous.
\newblock Deterministic and stochastic models for coalescence (aggregation and
  coagulation): a review of the mean-field theory for probabilists.
\newblock {\em Bernoulli}, 5(1):3--48, 1999.

\bibitem{AtNey}
K.~B. Athreya and P.~E. Ney.
\newblock {\em Branching processes}.
\newblock Dover Publications Inc., 2004.

\bibitem{BertoinTSS}
J.~Bertoin.
\newblock Two solvable systems of coagulation equations with limited
  aggregations.
\newblock {\em Ann. Inst. H. Poincar\'e Anal. Non Lin\'eaire},
  26(6):2073--2089, 2009.

\bibitem{BertoinSido}
J.~Bertoin and V.~Sidoravicius.
\newblock The structure of typical clusters in large sparse random
  configurations.
\newblock {\em J. Stat. Phys.}, 135(1):87--105, 2009.

\bibitem{DeaconuTanre}
M.~Deaconu and E.~Tanr{\'e}.
\newblock Smoluchowski's coagulation equation: probabilistic interpretation of
  solutions for constant, additive and multiplicative kernels.
\newblock {\em Ann. Scuola Norm. Sup. Pisa Cl. Sci. (4)}, 29(3):549--579, 2000.

\bibitem{DuboStew}
P.~B. Dubovski{\u\i} and I.~W. Stewart.
\newblock Existence, uniqueness and mass conservation for the
  coagulation-fragmentation equation.
\newblock {\em Math. Methods Appl. Sci.}, 19(7):571--591, 1996.

\bibitem{Dwass}
M.~Dwass.
\newblock The total progeny in a branching process and a related random walk.
\newblock {\em J. Appl. Probability}, 6:682--686, 1969.

\bibitem{ErnstKoG}
M.~H. Ernst, E.~M. Hendriks, and R.~M. Ziff.
\newblock Kinetics of gelation and universality.
\newblock {\em J. Phys. A}, 16(10):2293--2320, 1983.

\bibitem{ErnstCP}
M.~H. Ernst, E.~M. Hendriks, and R.~M. Ziff.
\newblock Coagulation processes with a phase transition.
\newblock {\em J. Colloid Interface Sci.}, 97:266--277, 1984.

\bibitem{EscobedoGaMC}
M.~Escobedo, P.~Lauren{\c{c}}ot, S.~Mischler, and B.~Perthame.
\newblock Gelation and mass conservation in coagulation-fragmentation models.
\newblock {\em J. Differential Equations}, 195(1):143--174, 2003.

\bibitem{EscobedoGiC}
M.~Escobedo, S.~Mischler, and B.~Perthame.
\newblock Gelation in coagulation and fragmentation models.
\newblock {\em Comm. Math. Phys.}, 231(1):157--188, 2002.

\bibitem{FouLauWP}
N.~Fournier and P.~Lauren{\c{c}}ot.
\newblock Well-posedness of {S}moluchowski's coagulation equation for a class
  of homogeneous kernels.
\newblock {\em J. Funct. Anal.}, 233(2):351--379, 2006.

\bibitem{FouLauMLP}
N.~Fournier and P.~Lauren{\c{c}}ot.
\newblock Marcus-{L}ushnikov processes, {S}moluchowski's and {F}lory's models.
\newblock {\em Stochastic Process. Appl.}, 119(1):167--189, 2009.

\bibitem{Jeon}
I.~Jeon.
\newblock Existence of gelling solutions for coagulation-fragmentation
  equations.
\newblock {\em Comm. Math. Phys.}, 194(3):541--567, 1998.

\bibitem{Kokholm}
N.~J. Kokholm.
\newblock On {S}moluchowski's coagulation equation.
\newblock {\em J. Phys. A}, 21(3):839--842, 1988.

\bibitem{LauGS}
P.~Lauren{\c{c}}ot.
\newblock Global solutions to the discrete coagulation equations.
\newblock {\em Mathematika}, 46(2):433--442, 1999.

\bibitem{LauCCC}
P.~Lauren{\c{c}}ot.
\newblock On a class of continuous coagulation-fragmentation equations.
\newblock {\em J. Differential Equations}, 167(2):245--274, 2000.

\bibitem{LauMi}
P.~Lauren{\c{c}}ot and S.~Mischler.
\newblock On coalescence equations and related models.
\newblock {\em In Degond, P., Pareschi, L. and Russo, G. (eds) : Modeling and
  computational methods for kinetic equations. Birkh\"auser}, pages 321--356,
  2004.

\bibitem{LeyvTschu}
F.~Leyvraz and H.~R. Tschudi.
\newblock Singularities in the kinetics of coagulation processes.
\newblock {\em J. Phys. A}, 14(12):3389--3405, 1981.

\bibitem{McLeod}
J.~B. McLeod.
\newblock On an infinite set of non-linear differential equations.
\newblock {\em Quart. J. Math. Oxford Ser. (2)}, 13:119--128, 1962.

\bibitem{MenonPego}
G.~Menon and R.~L. Pego.
\newblock Approach to self-similarity in {S}moluchowski's coagulation
  equations.
\newblock {\em Comm. Pure Appl. Math.}, 57(9):1197--1232, 2004.

\bibitem{NorMFC}
R.~Normand.
\newblock A model for coagulation with mating.
\newblock {\em Jour. Stat. Phys.}, 137(2):343--371, 2009.

\bibitem{NorrisReview}
J.~R. Norris.
\newblock Smoluchowski's coagulation equation: uniqueness, nonuniqueness and a
  hydrodynamic limit for the stochastic coalescent.
\newblock {\em Ann. Appl. Probab.}, 9(1):78--109, 1999.

\bibitem{NorrisCC}
J.~R. Norris.
\newblock Cluster coagulation.
\newblock {\em Comm. Math. Phys.}, 209(2):407--435, 2000.

\bibitem{StellZiff}
G.~Stell and R.~Ziff.
\newblock Kinetics of polymer gelation.
\newblock {\em J. Chem. Phys.}, 73:3492--3499, 1980.

\bibitem{vdHofstad}
R.~van~der Hofstad.
\newblock Random graphs and complex networks.
\newblock Available at \url{http://www.win.tue.nl/~rhofstad/NotesRGCN2010.pdf}.

\bibitem{ShirRoesSR}
H.~J. van Roessel and M.~Shirvani.
\newblock Some results on the coagulation equation.
\newblock {\em Nonlinear Anal.}, 43(5, Ser. A: Theory Methods):563--573, 2001.

\bibitem{ShirRoesPGM}
H.~J. van Roessel and M.~Shirvani.
\newblock A formula for the post-gelation mass of a coagulation equation with a
  separable bilinear kernel.
\newblock {\em Phys. D}, 222(1-2):29--36, 2006.

\bibitem{Smolu}
M.~von Smoluchowski.
\newblock Drei vortrage \"uber diffusion, brownsche molekularbewegung und
  koagulation von kolloidteilchen.
\newblock {\em Phys. Z.}, 17:557--571 and 585--599, 1916.

\bibitem{Genefunc}
H.~S. Wilf.
\newblock {\em Generatingfunctionology}.
\newblock Academic Press, 1994.
\newblock Also available online at
  \url{http://www.math.upenn.edu/~wilf/gfology2.pdf}.

\bibitem{ZiffKoP}
R.~M. Ziff.
\newblock Kinetics of polymerization.
\newblock {\em J. Statist. Phys.}, 23(2):241--263, 1980.

\end{thebibliography}
\end{document}